\documentclass[a4paper,12pt,reqno]{amsart}
\usepackage{amsfonts}
\usepackage{amsmath}
\usepackage{amssymb}
\usepackage[a4paper]{geometry}
\usepackage{mathrsfs}
\usepackage{hyperref}
\renewcommand\eqref[1]{(\ref{#1})}
\makeatletter
\newcommand*{\mint}[1]{%
  \mint@l{#1}{}%
}
\newcommand*{\mint@l}[2]{%
  \@ifnextchar\limits{%
    \mint@l{#1}%
  }{%
    \@ifnextchar\nolimits{%
      \mint@l{#1}%
    }{%
      \@ifnextchar\displaylimits{%
        \mint@l{#1}%
      }{%
        \mint@s{#2}{#1}%
      }%
    }%
  }%
}
\newcommand*{\mint@s}[2]{%
  \@ifnextchar_{%
    \mint@sub{#1}{#2}%
  }{%
    \@ifnextchar^{%
      \mint@sup{#1}{#2}%
    }{%
      \mint@{#1}{#2}{}{}%
    }%
  }%
}
\def\mint@sub#1#2_#3{%
  \@ifnextchar^{%
    \mint@sub@sup{#1}{#2}{#3}%
  }{%
    \mint@{#1}{#2}{#3}{}%
  }%
}
\def\mint@sup#1#2^#3{%
  \@ifnextchar_{%
    \mint@sup@sub{#1}{#2}{#3}%
  }{%
    \mint@{#1}{#2}{}{#3}%
  }%
}
\def\mint@sub@sup#1#2#3^#4{%
  \mint@{#1}{#2}{#3}{#4}%
}
\def\mint@sup@sub#1#2#3_#4{%
  \mint@{#1}{#2}{#4}{#3}%
}
\newcommand*{\mint@}[4]{%
  \mathop{}%
  \mkern-\thinmuskip
  \mathchoice{%
    \mint@@{#1}{#2}{#3}{#4}%
        \displaystyle\textstyle\scriptstyle
  }{%
    \mint@@{#1}{#2}{#3}{#4}%
        \textstyle\scriptstyle\scriptstyle
  }{%
    \mint@@{#1}{#2}{#3}{#4}%
        \scriptstyle\scriptscriptstyle\scriptscriptstyle
  }{%
    \mint@@{#1}{#2}{#3}{#4}%
        \scriptscriptstyle\scriptscriptstyle\scriptscriptstyle
  }%
  \mkern-\thinmuskip
  \int#1%
  \ifx\\#3\\\else_{#3}\fi
  \ifx\\#4\\\else^{#4}\fi
}
\newcommand*{\mint@@}[7]{%
  \begingroup
    \sbox0{$#5\int\m@th$}%
    \sbox2{$#5\int_{}\m@th$}%
    \dimen2=\wd0 %
    \let\mint@limits=#1\relax
    \ifx\mint@limits\relax
      \sbox4{$#5\int_{\kern1sp}^{\kern1sp}\m@th$}%
      \ifdim\wd4>\wd2 %
        \let\mint@limits=\nolimits
      \else
        \let\mint@limits=\limits
      \fi
    \fi
    \ifx\mint@limits\displaylimits
      \ifx#5\displaystyle
        \let\mint@limits=\limits
      \fi
    \fi
    \ifx\mint@limits\limits
      \sbox0{$#7#3\m@th$}%
      \sbox2{$#7#4\m@th$}%
      \ifdim\wd0>\dimen2 %
        \dimen2=\wd0 %
      \fi
      \ifdim\wd2>\dimen2 %
        \dimen2=\wd2 %
      \fi
    \fi
    \rlap{%
      $#5%
        \vcenter{%
          \hbox to\dimen2{%
            \hss
            $#6{#2}\m@th$%
            \hss
          }%
        }%
      $%
    }%
  \endgroup
}
\numberwithin{equation}{section}
\theoremstyle{plain}
\newtheorem{thm}{Theorem}[section]

\newtheorem{lem}[thm]{Lemma}

\theoremstyle{definition}
\newtheorem{defn}[thm]{Definition}
\newtheorem{rem}[thm]{Remark}

\begin{document}

   \title[Fractional Gagliardo-Nirenberg and weighted CKN inequalities]
{Anisotropic fractional Gagliardo-Nirenberg, weighted Caffarelli-Kohn-Nirenberg and Lyapunov-type  inequalities, and applications to Riesz potentials and $p$-sub-Laplacian systems}

\author[A. Kassymov]{Aidyn Kassymov}
\address{
  Aidyn Kassymov:
  \endgraf
  \endgraf
  Institute of Mathematics and Mathematical Modeling
  \endgraf
  125 Pushkin str.
  \endgraf
  050010 Almaty
  \endgraf
  Kazakhstan
  \endgraf
  and
  \endgraf
  Al-Farabi Kazakh National University
  \endgraf
   71 Al-Farabi avenue
   \endgraf
   050040 Almaty
   \endgraf
   Kazakhstan
  {\it E-mail address} {\rm kassymov@math.kz}
  }

   \author[M. Ruzhansky]{Michael Ruzhansky}
\address{
	Michael Ruzhansky:
	\endgraf
	Department of Mathematics
	\endgraf
	Imperial College London
	\endgraf
	180 Queen's Gate, London SW7 2AZ
	\endgraf
	United Kingdom
	\endgraf
	{\it E-mail address} {\rm m.ruzhansky@imperial.ac.uk}
}

\author[D. Suragan]{Durvudkhan Suragan}
\address{
	Durvudkhan Suragan:
	\endgraf
	Department of Mathematics
	\endgraf
	School of Science and Technology, Nazarbayev University
	\endgraf
	53 Kabanbay Batyr Ave, Astana 010000
	\endgraf
	Kazakhstan
	\endgraf
	{\it E-mail address} {\rm durvudkhan.suragan@nu.edu.kz}}

\thanks{
The authors were supported in parts by the EPSRC grant EP/R003025/1 and by the Leverhulme
Grant RPG-2017-151, as well as by the MESRK grant AP05130981.
\
}


     \begin{abstract}
 In this paper we prove the fractional Gagliardo-Nirenberg inequality on  homogeneous Lie groups. Also, we establish weighted fractional Caffarelli-Kohn-Nirenberg inequality and Lyapunov-type inequality for the Riesz potential on homogeneous Lie groups. The obtained Lyapunov inequality for the Riesz potential is new already in the classical setting of $\mathbb{R}^{N}$. As an application, we give two-sided estimate for the first eigenvalue of the Riesz potential. Also, we obtain Lyapunov inequality for the system of the fractional $p$-sub-Laplacian equations and give an application to estimate its eigenvalues.

     \end{abstract}
     \maketitle
   \tableofcontents
\section{Introduction}

 \subsection{Fractional Gagliardo-Nirenberg inequality}
 In the works of E. Gagliardo \cite{Gag} and L. Nirenberg \cite{Nir} (independently), they obtained the following (interpolation) inequality
 \begin{equation}
 \|u\|^{p}_{L^{p}(\mathbb{R}^{N})}\leq C \|\nabla u\|^{N(p-2)/2}_{L^{2}(\mathbb{R}^{N})}\|u\|^{(2p-N(p-2))/2}_{L^{2}(\mathbb{R}^{N})},\,\,u\in H^{1}(\mathbb{R}^{N}),
 \end{equation}
 where \begin{equation*}
 \begin{cases}
   2\leq p\leq\infty\,\,\text{for}\,\, N = 2,\\
   2\leq p \leq \frac{2N}{N-2}\,\,\text{for} \,\,N > 2.
 \end{cases}
\end{equation*}

The Gagliardo-Nirenberg inequality on the Heisenberg group $\mathbb{H}^{n}$ has the following form
\begin{equation}\label{heis}
\|u\|^{p}_{L^{p}(\mathbb{H}^{n})}\leq C \|\nabla_{\mathbb{H}^{n}} u\|^{Q(p-2)/2}_{L^{2}(\mathbb{H}^{n})}\|u\|^{(2p-Q(p-2))/2}_{L^{2}(\mathbb{H}^{n})},
\end{equation}
where $\nabla_{\mathbb{H}}$ is a horizontal gradient and $Q$ is a homogeneous dimension of  $\mathbb{H}^{n}$. In \cite{CR}, the authors established the best constant for the sub-elliptic Gagliardo-Nirenberg inequality \eqref{heis}. Consequently, in \cite{RTY} the best constants in Gagliardo-Nirenberg and Sobolev inequalities were also found for general hypoelliptic (Rockland operators) on general graded Lie groups.

In \cite{NS1} the authors obtained a fractional version of the Gagliardo-Nirenberg inequality in the following form:
\begin{equation}
\|u\|_{L^{\tau}(\mathbb{R}^{N})}\leq C[u]^{a}_{s,p}\|u\|^{1-a}_{L^{\alpha}(\mathbb{R}^{N})},\,\,\forall u\in C^{1}_{c}(\mathbb{R}^{N}),
\end{equation}
where $[u]_{s,p}$ is Gagliardo's seminorm defined by $$[u]_{s,p}^{p}=\int_{\mathbb{R}^{N}}\int_{\mathbb{R}^{N}}\frac{|u(x)-u(y)|^{p}}{|x-y|^{N+sp}}dxdy,$$ for $N\geq1,\,\,s\in(0,1),\,\,p>1,\,\,\alpha\geq1,\,\,\tau>0,$ and $a\in(0,1]$ is such that
\begin{equation*}
\frac{1}{\tau}=a\left(\frac{1}{p}-\frac{s}{N}\right)+\frac{1-a}{\alpha}.
\end{equation*}

In this paper we formulate the fractional Gagliardo-Nirenberg inequality on the homogeneous Lie groups.
To the best of our knowledge,  in this direction systematic studies on the homogeneous Lie groups started by the paper \cite{RSAM} in which homogeneous group versions of Hardy and Rellich inequalities were proved as consequences of universal identities.

 \subsection{Fractional Caffarelli-Kohn-Nirenberg inequality}
In their fundamental work \cite{CKN}, L. Caffarelli, R. Kohn and L. Nirenberg established:
\begin{thm}
Let $N\geq1$, and let $l_1$, $l_2$, $l_3$, $a, \, b, \, d,\, \delta \in \mathbb{R}$ be such that $l_1, l_2 \geq 1$,
$l_3 > 0, \,\,0 \leq \delta \leq 1,$ and
\begin{equation}
\frac{1}{l_1}+\frac{a}{N},\,\,\,\frac{1}{l_2}+\frac{b}{N},\,\,\,\frac{1}{l_3}+\frac{\delta d+(1-\delta) b}{N}>0.
\end{equation}
Then,
\begin{equation}
\||x|^{\delta d+(1-\delta) b}u\|_{L^{l_{3}}(\mathbb{R}^{N})}\leq C\||x|^{a}\nabla u\|^{\delta}_{L^{l_{1}}(\mathbb{R}^{N})}\||x|^{b} u\|^{1-\delta}_{L^{l_{2}}(\mathbb{R}^{N})},\,\,\,u\in C^{\infty}_{c}(\mathbb{R}^{N}),
\end{equation}
if and only if
\begin{multline}
\frac{1}{l_3}+\frac{\delta d+(1-\delta) b}{N}=\delta\left(\frac{1}{l_{1}}+\frac{a-1}{N}\right)+
(1-\delta)\left(\frac{1}{l_2}+\frac{b}{N}\right),\\
a-d\geq0,\,\,\,\,\text{if}\,\,\,\delta>0,\\
a-d\leq1,\,\,\,\,\text{if}\,\,\,\delta>0\,\,\,\text{and}\,\,\,\frac{1}{l_3}+\frac{\delta d+(1-\delta) b}{N}=\frac{1}{l_1}+\frac{a-1}{N},
\end{multline}
where $C$ is a positive constant independent of $u$.
\end{thm}

 In  \cite{NS1} the authors proved the fractional analogues of the Caffarelli-Kohn-Nirenberg inequality in weighted fractional Sobolev spaces. Also, in  \cite{Abd}  a fractional Caffarelli-Kohn-Nirenberg inequality for an admissible weight in $\mathbb{R}^{N}$ was obtained.

Recently many different versions of Caffarelli-Kohn-Nirenberg inequalities have been obtained, namely, in \cite{ZHD}  on the Heisenberg groups, in \cite{RSY3} and \cite{RS} on
stratified groups, in \cite{RSY1} and \cite{RSY2} on (general) homogeneous Lie groups.
 One of the aims of this paper is to prove the fractional weighted  Caffarelli-Kohn-Nirenberg inequality on the homogeneous Lie groups.
\subsection {Fractional Lyapunov-type inequality}
Historically, in Lyapunov's work \cite{Lyap} for the following one-dimensional homogeneous Dirichlet boundary value problem (for the second order ODE)

\begin{equation}\label{slayp}
 \begin{cases}
   u''(x)+\omega(x)u(x)=0,\,\,x\in(a,b),\\
   u(a)=u(b)=0,
 \end{cases}
\end{equation}
it was proved that if $u$ is a non-trivial solution of \eqref{slayp} and $\omega(x)$ is a real-valued and continuous function on $[a,b]$, then necessarily
 \begin{equation}\label{slaypin}
\int^{b}_{a}|\omega(x)|dx>\frac{4}{b-a}.
\end{equation}
 Nowadays, there are many extensions of Lyapunov's inequality. In \cite{Ed} the author obtains Lyapunov's inequality for the one-dimensional Dirichlet $p$-Laplacian
\begin{equation}\label{playp}
 \begin{cases}
   (|u'(x)|^{p-2}u'(x))'+\omega(x)u(x)=0,\,\,x\in(a,b),\,\,\,1<p<\infty,\\
   u(a)=u(b)=0,
 \end{cases}
\end{equation}
where $\omega(x)\in L^{1}(a,b)$, so necessarily
 \begin{equation}\label{playpin}
\int^{b}_{a}|\omega(x)|dx>\frac{2^{p}}{(b-a)^{p-1}},\,\,\,1<p<\infty.
\end{equation}
Obviously, taking $p=2$ in \eqref{playpin}, we recover the classical Lyapunov inequality \eqref{slaypin}.

In \cite{Kir} the authors obtained interesting results concerning Lyapunov inequalities for the multi-dimesional fractional $p$-Laplacian $(-\Delta_{p})^{s}$, $1<p<\infty,\,\,s\in(0,1)$, with a homogeneous Dirichlet boundary condition, that is,
\begin{equation}\label{fplayp}
 \begin{cases}
   (-\Delta_{p})^{s}u=\omega(x)|u|^{p-2}u,\,\,x\in\Omega,\\
   u(x)=0,\,\,x\in \mathbb{R}^{N}\setminus\Omega,
 \end{cases}
\end{equation}
where $\Omega\subset \mathbb{R}^{N}$ is an open set, $1<p<\infty,$ and $s\in(0,1).$
Let us recall the following result of \cite{Kir}.
\begin{thm}
Let $\Omega\subset \mathbb{R}^{N}$ be an open set, and let $\omega\in L^{\theta}(\Omega)$ with $1<\frac{N}{sp}<\theta<\infty,$  be a non-negative weight. Suppose that problem \eqref{fplayp} has a non-trivial weak solution $u\in W^{s,p}_{0}(\Omega)$. Then
\begin{equation}\label{fplyapq>sp}
\left(\int_{\Omega}\omega^{\theta}(x) \,dx \right)^{\frac{1}{\theta}}>\frac{C}{r_{\Omega}^{sp-\frac{N}{\theta}}},
\end{equation}
where $C>0$ is a universal constant and $r_{\Omega}$ is the inner radius of $\Omega$.
\end{thm}
In \cite{DP}, the authors considered a system of ODE for $p$ and $q$-Laplacian on the interval $(a,b)$ with the homogeneous Dirichlet condition in the following form:

\begin{equation}\label{ODEsys}
 \begin{cases}
 -(|u'(x)|^{p-2}u'(x))'=f(x)|u(x)|^{\alpha-2}u(x)|v(x)|^{\beta},\\
 -(|v'(x)|^{q-2}v'(x))'=g(x)|u(x)|^{\alpha}|v(x)|^{\beta-2}v(x),
    \end{cases}
\end{equation}
on the interval $(a,b),$ with
\begin{equation}\label{ODEDIR}
u(a)=u(b)=v(a)=v(b)=0,
\end{equation}
where  $f,g\in L^{1}(a,b),$ $f,g\geq0$, $p,q>1,$ $\alpha,\,\beta\geq0$ and
$$\frac{\alpha}{p}+\frac{\beta}{q}=1.$$ Then we have Lyapunov-type inequality for system \eqref{ODEsys} with homogeneous Dirichlet condition \eqref{ODEDIR}:
\begin{equation}
2^{\alpha+\beta}\leq (b-a)^{\frac{\alpha}{p'}+\frac{\beta}{q'}}\left(\int_{a}^{b}f(x)dx\right)^{\frac{\alpha}{p}}\left(\int_{a}^{b}g(x)dx\right)^{\frac{\beta}{q}},
\end{equation}
where $p'=\frac{p}{p-1}$ and $q'=\frac{q}{q-1}.$ In \cite{Kir1}, the authors obtained the Lyapunov-type inequality for a fractional $p$-Laplacian system in an open bounded subset $\Omega\subset \mathbb{R}^{N}$ with homogeneous Dirichlet conditions.
 One of our goals in this paper is to extend the Lyapunov-type inequality for the Riesz potential  and for the fractional $p$-sub-Laplacian system on the homogeneous Lie groups. These results are given in Theorem \ref{Lyapunov} and \ref{thmlyap}. Also, we give applications of the Lyapunov-type inequality for the Riesz potential and  for fractional $p$-sub-Laplacian system  on the homogeneous Lie groups. To demonstrate our techniques we consider the Riesz potential in the Abelian case $(\mathbb{R}^{N},+)$ and give two side estimates of the first eigenvalue of the Riesz potential in the Abelian case $(\mathbb{R}^{N},+)$.

 Summarising our main results of the present paper, we prove the following facts:
\begin{itemize}
\item An analogue of the fractional Gagliardo-Nirenberg inequality on the homogeneous group $\mathbb{G}$;
\item An analogue of the fractional weighted  Caffarelli-Kohn-Nirenberg inequality on $\mathbb{G}$;
\item An analogue of the Lyapunov-type inequality for the Riesz potential on $\mathbb{G}$;
\item An analogue of the Lyapunov-type inequality for the fractional $p$-sub-Laplacian system on $\mathbb{G}$.
\end{itemize}
The paper is organised as follows. First we give some basic discussions on fractional Sobolev spaces and related facts  on homogeneous Lie groups, then in Section \ref{SEC:3} we present the fractional Gagliardo-Nirenberg inequality on $\mathbb{G}$. The fractional weighted Caffarelli-Kohn-Nirenberg inequality on $\mathbb{G}$ is proved  in Section \ref{SEC:4}. In Section \ref{SEC:5} we discuss analogues of the Lyapunov-type inequalities for the Riesz potential and fractional $p$-sub-Laplacian system on $\mathbb{G}$.

\section{Preliminaries}
\label{SEC:2}
 We recall that a Lie group (on $\mathbb{R}^{n}$) $\mathbb{G}$ with the dilation
$$D_{\lambda}(x):=(\lambda^{\nu_{1}}x_{1},\ldots,\lambda^{\nu_{n}}x_{n}),\; \nu_{1},\ldots, \nu_{n}>0,\; D_{\lambda}:\mathbb{R}^{n}\rightarrow\mathbb{R}^{n},$$
which is an automorphism of the group $\mathbb{G}$ for each $\lambda>0,$
is called a {\em homogeneous (Lie) group}. In this paper, for simplicity, we use the notation $\lambda x$ instead of the dilation $D_{\lambda}(x)$. The homogeneous dimension of the homogeneous group $\mathbb{G}$ is denoted by $$Q:=\nu_{1}+\ldots+\nu_{n}.$$

A homogeneous quasi-norm on $\mathbb{G}$ is a continuous non-negative function
\begin{equation}
\mathbb{G}\ni x\mapsto q(x)\in[0,\infty),
\end{equation}
with the properties

\begin{itemize}
	\item[i)] $q(x)=q(x^{-1})$ for all $x\in\mathbb{G}$,
	\item[ii)] $q(\lambda x)=\lambda q(x)$ for all $x\in \mathbb{G}$ and $\lambda>0$,
	\item[iii)] $q(x)=0$ iff $x=0$.
\end{itemize}
Moreover, the following polarisation formula on homogeneous Lie groups will be used in our proofs:
there is a (unique)
positive Borel measure $\sigma$ on the
unit quasi-sphere
$
\omega_{Q}:=\{x\in \mathbb{G}:\,q(x)=1\},
$
so that for every $f\in L^{1}(\mathbb{G})$ we have
\begin{equation}\label{EQ:polar}
\int_{\mathbb{G}}f(x)dx=\int_{0}^{\infty}
\int_{\omega_{Q}}f(ry)r^{Q-1}d\sigma(y)dr.
\end{equation}
We refer to \cite{FS1} for the original appearance of such groups, and to \cite{FR} for a recent comprehensive treatment.  Let $p>1$, $s\in(0,1)$, and let $\mathbb{G}$ be a homogeneous Lie group of homogeneous dimension $Q$.
For a measurable function $u:\mathbb{G}\rightarrow \mathbb{R}$ we define the Gagliardo quasi-seminorm by
\begin{equation}\label{gsmnm}
[u]_{s,p,q}=\left( \int_{\mathbb{G}} \int_{\mathbb{G}}\frac{|u(x)-u(y)|^{p}}{q^{Q+sp}(y^{-1}\circ x)}dxdy\right)^{1/p}.
\end{equation}
Now we recall the definition of the fractional Sobolev spaces on homogeneous Lie groups denoted
by $W^{s,p,q}(\mathbb{G})$. For $p\geq1$ and $s\in(0,1)$, the functional space
\begin{equation}
W^{s,p,q}(\mathbb{G})=\{u\in L^{p}(\mathbb{G}): u \text {\,\,is measurable}, [u]_{s,p,q}<+\infty \},
\end{equation}
is called the fractional Sobolev space on $\mathbb{G}$.

Similarly, if $\Omega\subset \mathbb{G}$ is a Haar measurable set, we define the Sobolev space
\begin{multline}\label{ofss}
W^{s,p,q}(\Omega)=\{u\in L^{p}(\Omega): u \text {\,\,is measurable},\\
[u]_{s,p,q,\Omega}=\left( \int_{\Omega}\int_{\Omega}\frac{|u(x)-u(y)|^{p}}{q^{Q+sp}(y^{-1}\circ x)}dxdy\right)^{\frac{1}{p}}<+\infty\}.
\end{multline}

Now we recall the definition of the weighted fractional Sobolev space on the homogeneous Lie groups denoted by
\begin{multline}\label{wfss1}
W^{s,p,\beta,q}(\mathbb{G})=\{u\in L^{p}(\mathbb{G}): u \text {\,\,is measurable}, \\ [u]_{s,p,\beta,q}=\left(\int_{\mathbb{G}}\int_{\mathbb{G}}\frac{q^{\beta_{1}p}(x)q^{\beta_{2}p}(y)|u(x)-u(y)|^{p}}{q^{Q+sp}(y^{-1}\circ x)}dxdy\right)^{\frac{1}{p}}<+\infty \},
\end{multline}
where $\beta_{1},\,\,\beta_{2}\in \mathbb{R}$ with $\beta=\beta_{1}+\beta_{2}$ and it depends on $\beta_{1}$ and $\beta_{2}$.

As above, for a Haar measurable set $\Omega\subset \mathbb{G}$ , $p\geq1$, $s\in(0,1)$ and $\beta_{1},\,\,\beta_{2}\in \mathbb{R}$ with $\beta=\beta_{1}+\beta_{2}$,  we define the weighted fractional Sobolev space
\begin{multline}\label{wfss}
W^{s,p,\beta,q}(\Omega)=\{u\in L^{p}(\Omega): u \text {\,\,is measurable}, \\ [u]_{s,p,\beta,q,\Omega}=\left(\int_{\Omega}\int_{\Omega}\frac{q^{\beta_{1}p}(x)q^{\beta_{2}p}(y)|u(x)-u(y)|^{p}}{q^{Q+sp}(y^{-1}\circ x)}dxdy\right)^{\frac{1}{p}}<+\infty \}.
\end{multline}
Obviously, taking $\beta=\beta_{1}=\beta_{2}=0$ in \eqref{wfss}, we recover \eqref{ofss}.

The mean of a function $u$ is defined by
\begin{equation}
u_{\Omega}=\mint{-}_{\Omega}udx=\frac{1}{|\Omega|}\int_{\Omega}udx,\,\,\,u\in L^{1}(\Omega),
\end{equation}
where $|\Omega|$ is the Haar measure of $\Omega\subset \mathbb{G}$.

We will also use the decomposition of $\mathbb{G}$ into quasi-annuli $A_{k,q}$ defined by
\begin{equation}\label{Ak}
A_{k,q}:=\{x\in\mathbb{G}:\,\,2^{k}\leq q(x)<2^{k+1}\},
\end{equation}
where $q(x)$ is a quasi-norm on $\mathbb{G}$.

\section{Fractional Gagliargo-Nirenberg inequality on $\mathbb{G}$}
\label{SEC:3}
In this section we prove an analogue of the fractional Gagliardo-Nirenberg inequality on the homogeneous Lie groups. To prove Gagliardo-Nirenberg's inequality we need some preliminary results from \cite{KS}, a version of a fractional Sobolev inequality on the homogeneous Lie groups.

From now on, unless specified otherwise, $\mathbb{G}$ will be a homogeneous group of homogeneous dimension $Q$.

\begin{thm}[\cite{KS}, Fractional Sobolev inequality]\label{sob}
Let $p>1$, $s\in(0,1)$, $Q>sp,$ and let $q(\cdot)$ be a quasi-norm on $\mathbb{G}$. For any measurable and compactly supported function $u:\mathbb{G}\rightarrow \mathbb{R}$ there exists a positive constant $C=C(Q,p,s,q)>0$ such that
  \begin{equation}\label{sobin1}
  ||u||^{p}_{L^{p^{*}}(\mathbb{G})}\leq C[u]^{p}_{s,p,q},
  \end{equation}
  where $p^{*}=p^{*}(Q,s)=\frac{Qp}{Q-sp}$.
\end{thm}
\begin{thm}\label{GN}
Assume that $Q\geq2$, $s\in(0,1)$, $p>1$, $\alpha\geq1$, $\tau>0$, $a\in(0,1]$, $Q>sp$  and
\begin{equation*}
\frac{1}{\tau}=a\left(\frac{1}{p}-\frac{s}{Q}\right)+\frac{1-a}{\alpha}.
\end{equation*}
Then,
\begin{equation}
\|u\|_{L^{\tau}(\mathbb{G})}\leq C[u]^{a}_{s,p,q}\|u\|^{1-a}_{L^{\alpha}(\mathbb{G})},\,\,\forall\,\, u\in C^{1}_{c}(\mathbb{G}),
\end{equation}
 where $C=C(s,p,Q,a,\alpha)>0$.
\end{thm}
\begin{proof}[Proof of Theorem \ref{GN}]
By using the H\"{o}lder inequality, for every $\frac{1}{\tau}=a\left(\frac{1}{p}-\frac{s}{Q}\right)+\frac{1-a}{\alpha}$ we get
\begin{equation}\label{h1}
\|u\|^{\tau}_{L^{\tau}(\mathbb{G})}=\int_{\mathbb{G}}|u|^{\tau}dx=\int_{\mathbb{G}}|u|^{a\tau}|u|^{(1-a)\tau}dx\leq\|u\|_{L^{p^{*}}(\mathbb{G})}^{a\tau}\|u\|^{(1-a)\tau}_{L^{\alpha}(\mathbb{G})},
\end{equation}
where $p^{*}=\frac{Qp}{Q-sp}$.
From \eqref{h1}, by using the fractional Sobolev inequality (Theorem \ref{sob}), we obtain

\begin{equation*}
\|u\|^{\tau}_{L^{\tau}(\mathbb{G})}\leq\|u\|_{L^{p^{*}}(\mathbb{G})}^{a\tau}\|u\|^{(1-a)\tau}_{L^{\alpha}(\mathbb{G})}\leq C[u]^{a\tau}_{s,p,q}\|u\|^{(1-a)\tau}_{L^{\alpha}(\mathbb{G})},
\end{equation*}
that is,
\begin{equation}
\|u\|_{L^{\tau}(\mathbb{G})}\leq C[u]^{a}_{s,p,q}\|u\|^{1-a}_{L^{\alpha}(\mathbb{G})},
\end{equation}
where $C$ is a positive constant independent of $u$.
Theorem \ref{GN} is proved.
\end{proof}
\begin{rem}
In the Abelian case $(\mathbb{R}^{N},+)$ with the standard Euclidean distance instead of the quasi-norm, from Theorem \ref{GN} we get the fractional Gagliardo-Nirenberg inequality which was proved in \cite{NS1}.
\end{rem}

\section{Weighted fractional Caffarelli-Kohn-Nirenberg inequality on $\mathbb{G}$}\label{SEC:4}
In this section we prove the weighted fractional Caffarelli-Kohn-Nirenberg inequality on the homogeneous Lie groups.

\begin{thm}\label{CKN1}
Assume that $Q\geq2$, $s\in(0,1)$, $p>1$, $\alpha\geq1$, $\tau>0$, $a\in(0,1]$, $\beta_1,\,\beta_2,\,\beta,\,\mu,\,\gamma\in\mathbb{R}$, $\beta_1+\beta_2=\beta$  and
\begin{equation}\label{1.2}
\frac{1}{\tau}+\frac{\gamma}{Q}=a\left(\frac{1}{p}+\frac{\beta-s}{Q}\right)+(1-a)\left(\frac{1}{\alpha}+\frac{\mu}{Q}\right).
\end{equation}
Assume in addition that, $0\leq\beta-\sigma$ with
 $\gamma=a\sigma+(1-a)\mu,$
and
\begin{equation}\label{4.2}
\beta-\sigma\leq s\,\,\,\text{only if}\,\,\,\,\,\frac{1}{\tau}+\frac{\gamma}{Q}=\frac{1}{p}+\frac{\beta-s}{Q}.
\end{equation}
Then for $u\in C_{c}^{1}(\mathbb{G})$ we have
\begin{equation}\label{1in}
\|q^{\gamma}(x)u\|_{L^{\tau}(\mathbb{G})}\leq C [u]^{a}_{s,p,\beta,q}\|q^{\mu}(x)u\|^{1-a}_{L^{\alpha}(\mathbb{G})},
\end{equation}
when $\frac{1}{\tau}+\frac{\gamma}{Q}>0$, and
 for $u\in C_{c}^{1}(\mathbb{G}\setminus\{e\})$ we have
\begin{equation}\label{2in}
\|q^{\gamma}(x)u\|_{L^{\tau}(\mathbb{G})}\leq C [u]^{a}_{s,p,\beta,q}\|q^{\mu}(x)u\|^{1-a}_{L^{\alpha}(\mathbb{G})},
\end{equation}
when $\frac{1}{\tau}+\frac{\gamma}{Q}<0$.
Here  $e$ is the identity element of $\mathbb{G}$.
\end{thm}
\begin{rem}
	In the Abelian case $(\mathbb{R}^{N},+)$ with the standard Euclidean distance instead of quasi-norm in Theorem \ref{CKN1}, we get the (Euclidean) fractional Caffarelli-Kohn-Nirenberg inequality (see, e.g. \cite{NS1}, Theorem 1.1).
\end{rem}

To prove the fractional weighted Caffarelli-Kohn-Nirenberg inequality on $\mathbb{G}$ we will use Theorem \ref{GN} in the proof of the following lemma.

\begin{lem}\label{GN1}
	Assume that $Q\geq2$, $s\in(0,1)$, $p>1$, $\alpha\geq1$, $\tau>0$, $a\in(0,1]$  and
	\begin{equation*}
	\frac{1}{\tau}\geq a\left(\frac{1}{p}-\frac{s}{Q}\right)+\frac{1-a}{\alpha}.
	\end{equation*}
	Let $\lambda>0$ and $0<r<R$ and set
	\begin{equation*}
	\Omega=\{x\in\mathbb{G}:\,\,\lambda r<q(x)<\lambda R\}.
	\end{equation*}
	Then, for every $u\in C^{1}(\overline{\Omega})$, we have
	\begin{equation}\label{scal}
	\left(\mint{-}_{\Omega}|u-u_{\Omega}|^{\tau}dx\right)^{\frac{1}{\tau}}\leq C_{r,R}\lambda^{\frac{a(sp-Q)}{p}}[u]_{s,p,q,\Omega}^{a}\left(\mint{-}_{\Omega}|u|^{\alpha}dx\right)^{\frac{1-a}{\alpha}},
	\end{equation}
	where $C_{r,R}$ is a positive constant independent of $u$ and $\lambda$.
\end{lem}
\begin{proof}[Proof of Lemma \ref{GN1}]
	Without loss of generality, we assume that   $0 < s' \leq s$ and $\tau'  \geq \tau$ are such that
	\begin{equation*}
	\frac{1}{\tau'}= a\left(\frac{1}{p}-\frac{s'}{Q}\right)+\frac{1-a}{\alpha},
	\end{equation*}
and $\lambda=1$, then let $\Omega_{1}$ be
\begin{equation*}
	\Omega_{1}=\{x\in\mathbb{G}:\,\, r<q(x)< R\}.
	\end{equation*}
	By using Theorem \ref{GN}, Jensen's inequality and $[u]_{s',p,q,\Omega}\leq C [u]_{s,p,q,\Omega}$, we get
\begin{multline}\label{4.6}	
\left(\mint{-}_{\Omega_{1}}|u-u_{\Omega_{1}}|^{\tau}dx\right)^{\frac{1}{\tau}}=\frac{1}{|\Omega_{1}|^{\frac{1}{\tau}}}\|u-u_{\Omega_{1}}\|_{\tau}\leq C_{r,R}\|u-u_{\Omega_{1}}\|_{L^{\tau'}(\Omega_{1})}\\
\leq C_{r,R} [u-u_{\Omega_{1}}]_{s',p,q,\Omega_{1}}^{a}\|u\|^{1-a}_{L^{\alpha}(\Omega_{1})}\\
\leq C_{r,R} \left(\int_{\Omega_{1}}\int_{\Omega_{1}}\frac{|u(x)-u_{\Omega_{1}}-u(y)+u_{\Omega_{1}}|^{p}}{q^{Q+s'p}(y^{-1}\circ x)}dxdy\right)^{\frac{a}{p}}\|u\|^{1-a}_{L^{\alpha}(\Omega_{1})}\\
\leq C_{r,R}[u]^{a}_{s,p,q,\Omega_{1}}\|u\|^{1-a}_{L^{\alpha}(\Omega_{1})}\leq C_{r,R}[u]^{a}_{s,p,q,\Omega_{1}}\left(\mint{-}_{\Omega_{1}}|u|^{\alpha}dx\right)^{\frac{1-a}{\alpha}},
\end{multline}
where $C_{r,R}>0$. Let us set $u(\lambda x)$ instead of $u(x)$, then
\begin{multline}	
\left(\mint{-}_{\Omega_{1}}\left|u(\lambda x)-\mint{-}_{\Omega_{1}}u(\lambda x)dx\right|^{\tau}dx\right)^{\frac{1}{\tau}} \leq C_{r,R} \left(\int_{\Omega_{1}}\int_{\Omega_{1}}\frac{|u(\lambda x)-u(\lambda y)|^{p}}{q^{Q+sp}(y^{-1}\circ x)}dxdy\right)^{\frac{a}{p}}\\ \times\left(\frac{1}{|\Omega_{1}|}\int_{\Omega_{1}}|u(\lambda x)|^{\alpha}dx\right)^{\frac{1-a}{\alpha}}.
\end{multline}
Thus, we compute
\begin{multline}	
\left(\mint{-}_{\Omega}\left|u(x)-\mint{-}_{\Omega}u(x)dx\right|^{\tau}dx\right)^{\frac{1}{\tau}}=\left(\frac{1}{|\Omega|}\int_{\Omega}\left|u(x)-\frac{1}{|\Omega|}\int_{\Omega}u(x)dx\right|^{\tau}dx\right)^{\frac{1}{\tau}}\\
=\left(\frac{1}{|\Omega|}\int_{\Omega}\left|u(\lambda y)-\frac{1}{|\Omega|}\int_{\Omega}u(\lambda y)d(\lambda y)\right|^{\tau}d(\lambda y)\right)^{\frac{1}{\tau}}\\
=\left(\frac{1}{|\Omega_{1}|}\int_{\Omega_{1}}\frac{\lambda^{Q}}{\lambda^{Q}}\left|u(\lambda y)-\frac{\lambda^{Q}}{\lambda^{Q}|\Omega_{1}|}\int_{\Omega_{1}}u(\lambda y)d y\right|^{\tau}d y\right)^{\frac{1}{\tau}}\\
=\left(\frac{1}{|\Omega_{1}|}\int_{\Omega_{1}}\left|u(\lambda y)-\frac{1}{|\Omega_{1}|}\int_{\Omega_{1}}u(\lambda y)dy\right|^{\tau}dy\right)^{\frac{1}{\tau}}\\
\leq C_{r,R} \left(\int_{\Omega_{1}}\int_{\Omega_{1}}\frac{|u(\lambda x)-u(\lambda y)|^{p}}{q^{Q+sp}(y^{-1}\circ x)}dxdy\right)^{\frac{a}{p}}\left(\frac{1}{|\Omega_{1}|}\int_{\Omega_{1}}|u(\lambda x)|^{\alpha}dx\right)^{\frac{1-a}{\alpha}}\\
= C_{r,R} \left(\int_{\Omega_{1}}\int_{\Omega_{1}}\frac{\lambda^{2Q}\lambda^{Q+sp}|u(\lambda x)-u(\lambda y)|^{p}}{\lambda^{2Q}\lambda^{Q+sp}q^{Q+sp}(y^{-1}\circ x)}dxdy\right)^{\frac{a}{p}}\left(\frac{1}{|\Omega_{1}|}\int_{\Omega_{1}}\frac{\lambda^{Q}}{\lambda^{Q}}|u(\lambda x)|^{\alpha}dx\right)^{\frac{1-a}{\alpha}}\\
=C_{r,R} \left(\int_{\Omega}\int_{\Omega}\frac{\lambda^{sp-Q}|u(\lambda x)-u(\lambda y)|^{p}}{q^{Q+sp}((\lambda y)^{-1}\circ \lambda x)}d(\lambda x)d (\lambda y)\right)^{\frac{a}{p}}\left(\frac{1}{|\Omega|}\int_{\Omega}|u(\lambda x)|^{\alpha}d(\lambda x)\right)^{\frac{1-a}{\alpha}}\\
=C_{r,R} \left(\int_{\Omega}\int_{\Omega}\frac{\lambda^{sp-Q}|u(x)-u(y)|^{p}}{q^{Q+sp}(y^{-1}\circ x)}dxd y\right)^{\frac{a}{p}}\left(\frac{1}{|\Omega|}\int_{\Omega}|u(x)|^{\alpha}dx\right)^{\frac{1-a}{\alpha}}\\
=C_{r,R}\lambda^{\frac{a(sp-Q)}{p}}[u]^{a}_{s,p,q,\Omega}\left(\frac{1}{|\Omega|}\int_{\Omega}|u(x)|^{\alpha}dx\right)^{\frac{1-a}{\alpha}}.
\end{multline}
The proof of Lemma \ref{GN1} is complete.
\end{proof}

\begin{proof}[Proof of Theorem \ref{CKN1}]

First let us consider the case \eqref{4.2}, that is, $\beta-\sigma\leq s$ and $\frac{1}{\tau}+\frac{\gamma}{Q}=\frac{1}{p}+\frac{\beta-s}{Q}$. By using Lemma \ref{GN1} with $\lambda=2^{k}$, $r=1,$ $R=2$ and $\Omega=A_{k,q}$, we get
\begin{equation}\label{2.3}
\left(\mint{-}_{A_{k,q}}|u-u_{A_{k,q}}|^{\tau}dx\right)^{\frac{1}{\tau}}
\leq C 2^{\frac{ak(sp-Q)}{p}}[u]_{s,p,q,A_{k,q}}^{a}\left(\mint{-}_{A_{k,q}}|u|^{\alpha}dx\right)^{\frac{1-a}{\alpha}},
\end{equation}
where $A_{k,q}$ is defined in \eqref{Ak} and $k\in\mathbb{Z}$. Now by using \eqref{2.3} we obtain
$$\int_{A_{k,q}}|u|^{\tau}dx=\int_{A_{k,q}}|u-u_{A_{k,q}}+u_{A_{k,q}}|^{\tau}dx\\
\leq C\left(\int_{A_{k,q}}|u_{A_{k,q}}|^{\tau}dx+\int_{A_{k,q}}|u-u_{A_{k,q}}|^{\tau}dx\right)$$
\begin{multline}\label{ocenka}
 =C\left(\int_{A_{k,q}}|u_{A_{k,q}}|^{\tau}dx+\frac{|A_{k,q}|}{|A_{k,q}|}\int_{A_{k,q}}|u-u_{A_{k,q}}|^{\tau}dx\right)\\
 =C\left(|A_{k,q}||u_{A_{k,q}}|^{\tau}+|A_{k,q}|\mint{-}_{A_{k,q}}|u-u_{A_{k,q}}|^{\tau}dx\right)\\
 \leq C\left(|A_{k,q}||u_{A_{k,q}}|^{\tau}+2^{\frac{ak(sp-Q)\tau}{p}}|A_{k,q}|[u]_{s,p,q,A_{k,q}}^{a\tau}\left(\frac{1}{|A_{k,q}|}\int_{A_{k,q}}|u|^{\alpha}dx\right)^{\frac{(1-a)\tau}{\alpha}}\right)\\
  \leq C \left(2^{Q k}|u_{A_{k,q}}|^{\tau}+2^{\frac{ak(sp-Q)\tau}{p}}2^{kQ}2^{-\frac{Q(1-a)\tau k}{\alpha}}[u]_{s,p,q,A_{k,q}}^{a\tau}\|u\|^{(1-a)\tau}_{L^{\alpha}(A_{k,q})}\right).
\end{multline}
Then, from \eqref{ocenka} we get
\begin{multline}
\int_{A_{k,q}}q^{\gamma\tau}(x)|u|^{\tau}dx\leq2^{(k+1)\gamma\tau}\int_{A_{k,q}}|u|^{\tau}dx\leq C2^{(Q+\gamma\tau)k}|u_{A_{k,q}}|^{\tau}\\
+C2^{\gamma\tau k}2^{kQ}2^{\frac{ak(sp-Q)\tau}{p}}2^{-\frac{Q(1-a)\tau k}{\alpha}}[u]^{a\tau}_{s,p,q,A_{k,q}}\|u\|^{(1-a)\tau}_{L^{\alpha}(A_{k,q})}=C2^{(Q+\gamma\tau)k}|u_{A_{k,q}}|^{\tau}\\
+C2^{\left(\gamma\tau+Q+\frac{a(sp-Q)\tau}{p}-\frac{Q(1-a)\tau}{\alpha}\right) k}\left(\int_{A_{k,q}}\int_{A_{k,q}}\frac{2^{kp\beta_{1}}2^{kp\beta_{2}}|u(x)-u(y)|^{p}}{2^{kp\beta}q^{Q+sp}(y^{-1}\circ x)}dxdy\right)^{\frac{a\tau}{p}}\\
\times\left(\int_{A_{k,q}}\frac{2^{k\alpha\mu}}{2^{k\alpha\mu}}|u(x)|^{\alpha}dx\right)^{\frac{(1-a)\tau}{\alpha}}\leq C2^{(Q+\gamma\tau)k}|u_{A_{k,q}}|^{\tau}\\
+C2^{\left(\gamma\tau+Q+\frac{a(sp-Q)\tau}{p}-\frac{Q(1-a)\tau}{\alpha}-a\beta\tau-\mu\tau(1-a)\right) k}\left(\int_{A_{k,q}}\int_{A_{k,q}}\frac{q^{p\beta_{1}}(x)q^{p\beta_{2}}(y)|u(x)-u(y)|^{p}}{q^{Q+sp}(y^{-1}\circ x)}dxdy\right)^{\frac{a\tau}{p}}\\
\times\left(\int_{A_{k,q}}q^{\alpha\mu}(x)|u(x)|^{\alpha}dx\right)^{\frac{(1-a)\tau}{\alpha}}\leq C2^{(Q+\gamma\tau)k}|u_{A_{k,q}}|^{\tau}\\
+C2^{\left(\gamma\tau+Q+\frac{a(sp-Q)\tau}{p}-\frac{Q(1-a)\tau}{\alpha}-a\beta\tau-\mu\tau(1-a)\right) k}[u]_{s,p,\beta,q,A_{k,q}}^{a\tau}\|q^{\mu}(x)u\|^{(1-a)\tau}_{L^{\alpha}(A_{k,q})}.
\end{multline}
Here by \eqref{1.2}, we have
\begin{multline}
\gamma\tau+Q+\frac{a(sp-Q)\tau}{p}-\frac{Q(1-a)\tau}{\alpha}-a\beta\tau-\mu\tau(1-a)\\
=Q\tau\left(\frac{\gamma}{Q}+\frac{1}{\tau}+\frac{a(sp-Q)}{Qp}-\frac{(1-a)}{\alpha}-\frac{a\beta}{Q}-\frac{\mu(1-a)}{Q}\right)\\
=Q\tau\left(a\left(\frac{1}{p}+\frac{\beta-s}{Q}\right)+(1-a)\left(\frac{1}{\alpha}+\frac{\mu}{Q}\right)+\frac{a(sp-Q)}{Qp}-\frac{(1-a)}{\alpha}-\frac{a\beta}{Q}-\frac{\mu(1-a)}{Q}\right)\\
=0.
\end{multline}
Thus, we obtain
\begin{equation}\label{2.4}
\int_{A_{k,q}}q^{\gamma \tau}(x)|u|^{\tau}dx\leq C2^{(\gamma\tau+Q)k}|u_{A_{k,q}}|^{\tau}+C[u]^{a\tau}_{s,p,\beta,q,A_{k,q}}\|q^{\mu}(x)u\|^{(1-a)\tau}_{L^{\alpha}(A_{k,q})},
\end{equation}
and by summing over $k$ from $m$ to $n$, we get
  \begin{multline}\label{2.5}
\int_{\cup_{k=m}^{n} A_{k,q}}q^{\gamma \tau}(x)|u|^{\tau}dx=\int_{\{2^{m}<q(x)<2^{n+1}\}}q^{\gamma \tau}(x)|u|^{\tau}dx\leq C\sum_{k=m}^{n}2^{(\gamma\tau+Q)k}|u_{A_{k,q}}|^{\tau}\\
+C\sum_{k=m}^{n}[u]^{a\tau}_{s,p,\beta,q,A_{k,q}}\|q^{\mu}(x)u\|^{(1-a)\tau}_{L^{\alpha}(A_{k,q})},
\end{multline}
where $k,m,n\in\mathbb{Z}$ and $m\leq n-2$.

To prove \eqref{1in} let us choose $n$ such that
\begin{equation}\label{supp}
\text{supp}\,u\subset B_{2^{n}},
\end{equation}
where $B_{2^{n}}$ is a quasi-ball of $\mathbb{G}$ with the radius $2^{n}$.

The following known inequality will be used in the proof.
\begin{lem}[Lemma 2.2, \cite{NS2}]\label{elin}
	Let $\xi>1$ and $\eta>1$. Then exists a positive constant $C$ depending $\xi$ and $\eta$ such that $1<\zeta<\xi$,
	\begin{equation}
	(|a|+|b|)^{\eta}\leq \zeta|a|^{\eta}+\frac{C}{(\zeta-1)^{\eta-1}}|b|^{\eta},\,\,\,\,\,\forall\,\,a,b\in\mathbb{R}.
	\end{equation}
\end{lem}
Let us consider the following integral
\begin{multline*}
\mint{-}_{A_{k+1,q}\cup A_{k,q}}\left|u-\mint{-}_{A_{k+1,q}\cup A_{k,q}}u\right|^{\tau}dx\\=\frac{1}{|A_{k+1,q}|+|A_{k,q}|}\int_{A_{k+1,q}\cup A_{k,q}}\left|u-\mint{-}_{A_{k+1,q}\cup A_{k,q}}u\right|^{\tau}dx\\
=\frac{1}{|A_{k+1,q}|+|A_{k,q}|}\left(\int_{A_{k+1,q}}\left|u-\mint{-}_{A_{k+1,q}\cup A_{k,q}}u\right|^{\tau}dx+\int_{A_{k,q}}\left|u-\mint{-}_{A_{k+1,q}\cup A_{k,q}}u\right|^{\tau}dx\right).
\end{multline*}
On the other hand, a direct calculation gives
\begin{multline}\label{A_{k+1}-A_{k}}
\mint{-}_{A_{k+1,q}\cup A_{k,q}}\left|u-\mint{-}_{A_{k+1,q}\cup
A_{k,q}}u\right|^{\tau}dx\\
=\frac{1}{|A_{k+1,q}|+|A_{k,q}|}\left(\int_{A_{k+1,q}}\left|u-\mint{-}_{A_{k+1,q}\cup A_{k,q}}u\right|^{\tau}dx+\int_{A_{k,q}}\left|u-\mint{-}_{A_{k+1,q}\cup A_{k,q}}u\right|^{\tau}dx\right)\\
\geq\frac{1}{|A_{k+1,q}|+|A_{k,q}|}\int_{A_{k,q}}\left|u-\mint{-}_{A_{k+1,q}\cup A_{k,q}}u\right|^{\tau}dx\\
\geq\frac{1}{|A_{k+1,q}|+|A_{k,q}|}\left|\int_{A_{k,q}}\left(u-\mint{-}_{A_{k+1,q}\cup A_{k,q}}u\right)dx\right|^{\tau}\\
=\frac{1}{|A_{k+1,q}|+|A_{k,q}|}\left|\int_{A_{k,q}}udx-\frac{|A_{k,q}|}{|A_{k+1,q}|+|A_{k,q}|}\int_{A_{k,q}}udx-\frac{|A_{k,q}|}{|A_{k+1,q}|+|A_{k,q}|}\int_{A_{k+1,q}}udx\right|^{\tau}\\
=\frac{1}{|A_{k+1,q}|+|A_{k,q}|}\left|\frac{|A_{k+1,q}|}{|A_{k+1,q}|+|A_{k,q}|}\int_{A_{k,q}}udx-\frac{|A_{k,q}|}{|A_{k+1,q}|+|A_{k,q}|}\int_{A_{k+1,q}}udx\right|^{\tau}\\
=\frac{1}{(|A_{k+1,q}|+|A_{k,q}|)^{2}}\left||A_{k+1,q}|\int_{A_{k,q}}udx-|A_{k,q}|\int_{A_{k+1,q}}udx\right|^{\tau}\\
=\frac{|A_{k+1,q}||A_{k,q}|}{(|A_{k+1,q}|+|A_{k,q}|)^{2}}\left|\frac{1}{|A_{k,q}|}\int_{A_{k,q}}udx-\frac{1}{|A_{k+1,q}|}\int_{A_{k+1,q}}udx\right|^{\tau}\\
=\frac{|A_{k+1,q}||A_{k,q}|}{(|A_{k+1,q}|+|A_{k,q}|)^{2}}|u_{A_{k+1,q}}-u_{A_{k,q}}|^{\tau}\geq C\frac{2^{Qk}2^{Q(k-1)}}{(2^{Qk}-2^{Q(k-1)})^{2}}|u_{A_{k+1,q}}-u_{A_{k,q}}|^{\tau}\\
\geq C\frac{2^{2Qk}2^{-Q}}{2^{2kQ}(1+2^{-Q})^{2}}|u_{A_{k+1,q}}-u_{A_{k,q}}|^{\tau}\geq C |u_{A_{k+1,q}}-u_{A_{k,q}}|^{\tau}.
\end{multline}
From \eqref{A_{k+1}-A_{k}} and Lemma \ref{GN1}, we obtain
\begin{multline}
|u_{A_{k+1,q}}-u_{A_{k,q}}|^{\tau}\leq C\mint{-}_{A_{k+1,q}\cup A_{k,q}}\left|u-\mint{-}_{A_{k+1,q}\cup A_{k,q}}u\right|^{\tau}dx\\
\leq C2^{\frac{a k(sp-Q)}{p}}[u]_{s,p,q,A_{k+1,q}\cup A_{k,q}}^{\tau a}\left(\mint{-}_{A_{k+1,q}\cup A_{k,q}}|u|^{\alpha}dx\right)^{\frac{(1-a)\tau}{\alpha}}.
\end{multline}
By using this fact, taking $\tau=1$ we have
\begin{multline}
|u_{A_{k,q}}|\leq|u_{A_{k+1,q}}-u_{A_{k,q}}|+|u_{A_{k+1,q}}|\\
\leq |u_{A_{k+1,q}}| +C2^{\frac{a k(sp-Q)}{p}}[u]_{s,p,q,A_{k+1,q}\cup A_{k,q}}^{ a}\left(\mint{-}_{A_{k+1,q}\cup A_{k,q}}|u|^{\alpha}dx\right)^{\frac{(1-a)}{\alpha}},
\end{multline}
and by using Lemma \ref{elin} with $\eta=\tau$, $\zeta=2^{\gamma\tau+Q}c$, where $c=\frac{2}{1+2^{\gamma\tau+Q}}<1$, since $\gamma\tau+Q>0$, we have
\begin{equation*}
2^{(\gamma\tau+Q)k}|u_{A_{k,q}}|^{\tau}\leq c 2^{(k+1)(\gamma\tau+Q)}|u_{A_{k+1,q}}|^{\tau}+C [u]^{\tau a}_{s,p,\beta,q,A_{k+1,q}\cup A_{k,q}}\|q^{\mu}(x)u\|^{(1-a)\tau}_{L^{\alpha}(A_{k+1,q}\cup A_{k,q})}.
\end{equation*}
By summing over $k$ from $m$ to $n$ and by using $\eqref{supp}$ we have
\begin{multline}\label{4.20}
\sum_{k=m}^{n}2^{(\gamma\tau+Q)k}|u_{A_{k,q}}|^{\tau}\leq \sum_{k=m}^{n} c 2^{(k+1)(\gamma\tau+Q)}|u_{A_{k+1,q}}|^{\tau}\\
+C \sum_{k=m}^{n}[u]^{\tau a}_{s,p,\beta,q,A_{k+1,q}\cup A_{k,q}}\|q^{\mu}(x)u\|^{(1-a)\tau}_{L^{\alpha}(A_{k+1,q}\cup A_{k,q})}.
\end{multline}
By using \eqref{4.20}, we compute
\begin{multline}
(1-c)\sum_{k=m}^{n}2^{(\gamma\tau+Q)k}|u_{A_{k,q}}|^{\tau}\leq 2^{(\gamma\tau+Q)m}|u_{A_{m,q}}|^{\tau}+(1-c)\sum_{k=m+1}^{n}2^{(\gamma\tau+Q)k}|u_{A_{k,q}}|^{\tau} \\
\leq C \sum_{k=m}^{n}[u]^{\tau a}_{s,p,\beta,q,A_{k+1,q}\cup A_{k,q}}\|q^{\mu}(x)u\|^{(1-a)\tau}_{L^{\alpha}(A_{k+1,q}\cup A_{k,q})}.
\end{multline}
This yields 
\begin{equation}\label{2.6}
\sum_{k=m}^{n}2^{(\gamma\tau+Q)k}|u_{A_{k,q}}|^{\tau}\leq C\sum_{k=m}^{n}[u]^{\tau a}_{s,p,\beta,q,A_{k+1,q}\cup A_{k,q}}\|q^{\mu}(x)u\|^{(1-a)\tau}_{L^{\alpha}(A_{k+1,q}\cup A_{k,q})}.
\end{equation}
From \eqref{2.5} and \eqref{2.6}, we have
\begin{equation}\label{2.5.6}
\int_{\{2^{m}<q(x)<2^{n+1}\}}q^{\gamma \tau}(x)|u|^{\tau}dx\leq C \sum_{k=m}^{n}[u]^{\tau a}_{s,p,\beta,q,A_{k+1,q}\cup A_{k,q}}\|q^{\mu}(x)u\|^{(1-a)\tau}_{L^{\alpha}(A_{k+1,q}\cup A_{k,q})}.
\end{equation}
Let $s,t\geq0$ be such that $s+t\geq1$. Then for any $x_{k},y_{k}\geq0$, we have
\begin{equation}\label{kb}
\sum_{k=m}^{n}x_{k}^{s}y_{k}^{t}\leq
\left(\sum_{k=m}^{n}x_{k}\right)^{s}
\left(\sum_{k=m}^{n}y_{k}\right)^{t}.
\end{equation}
By using this inequality in \eqref{2.5.6} with $s=\frac{\tau a}{p}$, $t=\frac{(1-a)\tau}{\alpha}$, $\frac{a}{p}+\frac{1-a}{\alpha}\geq \frac{1}{\tau}$ and $s\geq\beta-\sigma$, we obtain
\begin{equation}
\int_{\{q(x)>2^{m}\}}q^{\gamma \tau}(x)|u|^{\tau}dx\leq C[u]^{a\tau}_{s,p,\beta,q,\cup_{k=m}^{\infty}A_{k,q}}\|q^{\mu}(x)u\|^{(1-a)\tau}_{L^{\alpha}(\cup_{k=m}^{\infty}A_{k,q})}.
\end{equation}
Inequality \eqref{1in} is proved.

Let us prove \eqref{2in}. The strategy of the proof is similar to the previous case.
Choose $m$ such that
\begin{equation}\label{supp2}
\text{supp}\, u\cap B_{2^{m}}=\emptyset.
\end{equation}
From Lemma \ref{GN1} we have
$$|u_{A_{k+1,q}}-u_{A_{k,q}}|^{\tau}\leq C2^{\frac{a\tau k(sp-Q)}{p}}[u]_{s,p,q,A_{k+1,q}\cup A_{k,q}}^{\tau a}\left(\mint{-}_{A_{k+1,q}\cup A_{k,q}}|u|^{\alpha}dx\right)^{\frac{(1-a)\tau}{\alpha}}.$$
By Lemma \ref{elin} and choosing $c=\frac{1+2^{\gamma\tau+Q}}{2}<1$, since $\gamma\tau+Q<0$, we have
\begin{equation*}
2^{(\gamma\tau+Q)(k+1)}|u_{A_{k+1,q}}|^{\tau}\leq c 2^{k(\gamma\tau+Q)}|u_{A_{k,q}}|^{\tau}+C [u]^{\tau a}_{s,p,\beta,q,A_{k+1,q}\cup A_{k,q}}\|q^{\mu}(x)u\|^{(1-a)\tau}_{L^{\alpha}(A_{k+1,q}\cup A_{k,q})},
\end{equation*}
and by summing over $k$ from $m$ to $n$  and by using \eqref{supp2} we obtain
\begin{equation}\label{2.9}
\sum_{k=m}^{n}2^{(\gamma\tau+Q)k}|u_{A_{k,q}}|^{\tau}\leq C\sum_{k=m-1}^{n-1}[u]^{\tau a}_{s,p,\beta,q,A_{k+1,q}\cup A_{k,q}}\|q^{\mu}(x)u\|^{(1-a)\tau}_{L^{\alpha}(A_{k+1,q}\cup A_{k,q})}.
\end{equation}
From \eqref{2.5} and \eqref{2.9}, we establish that
\begin{equation}\label{2.5.9}
\int_{\{2^{m}<q(x)<2^{n+1}\}}q^{\gamma \tau}(x)|u|^{\tau}dx\leq C \sum_{k=m-1}^{n-1}[u]^{\tau a}_{s,p,\beta,q,A_{k+1,q}\cup A_{k,q}}\|q^{\mu}(x)u\|^{(1-a)\tau}_{L^{\alpha}(A_{k+1,q}\cup A_{k,q})}.
\end{equation}
Now by using \eqref{kb} we get
\begin{equation}
\int_{\{q(x)<2^{n+1}\}}q^{\gamma \tau}(x)|u|^{\tau}dx\leq C [u]^{\tau a}_{s,p,\beta,q,\cup_{k=-\infty}^{n}A_{k,q}}\|q^{\mu}(x)u\|^{(1-a)\tau}_{L^{\alpha}(\cup_{k=-\infty}^{n}A_{k,q})}.
\end{equation}
The proof of the case $s\geq\beta-\sigma$ is complete.

Let us prove the case of $\beta-\sigma>s$. Without loss of generality, we assume that
\begin{equation}
[u]_{s,p,\beta,q}=\|u\|_{L^{\alpha}(\mathbb{G})}=1,
\end{equation}
where
$$\frac{1}{p}+\frac{\beta-s}{Q}\neq\frac{1}{\alpha}+\frac{\mu}{Q}.$$
We also assume that $a_{1}>0,\,\,1>a_{2}$ and $\tau_{1},\,\tau_{2}>0$ with
\begin{equation}\label{tau2}
\frac{1}{\tau_{2}}=\frac{a_{2}}{p}+\frac{1-a_{2}}{\alpha},
\end{equation}
and
\begin{multline}\label{tau1}
\,\,\,\,\,\,\,\,\,\,\,\,\,\,\,\,\,\,\,\,\,\,\,\,\,\,\,\,\,\,\,\,\text{if} \,\,\,\,\,\,\, \frac{a}{p}+\frac{1-a}{\alpha}-\frac{as}{Q}>0,\,\,\,\,\text{then}\,\,\,\frac{1}{\tau_{1}}=\frac{a_{1}}{p}+\frac{1-a_{1}}{\alpha}-\frac{a_{1}s}{Q},\\
\,\,\,\,\,\,\,\,\,\,\,\,\,\,\,\,\,\,\,\,\,\,\,\,\,\,\,\,\,\,\,\,\text{if} \,\,\,\,\,\,\,\frac{a}{p}+\frac{1-a}{\alpha}-\frac{as}{Q}\leq0,\,\,\,\,\text{then}\,\,\,\frac{1}{\tau}>\frac{1}{\tau_{1}}\geq\frac{a_{1}}{p}+\frac{1-a_{1}}{\alpha}-\frac{a_{1}s}{Q}.
\end{multline}
Taking $\gamma_{1}=a_{1}\beta+(1-a_{1})\mu$ and $\gamma_{2}=a_{2}(\beta-s)+(1-a_{2})\mu$, we obtain
\begin{equation}\label{2.11}
\frac{1}{\tau_{1}}+\frac{\gamma_{1}}{Q}\geq a_{1}\left(\frac{1}{p}+\frac{\beta-s}{Q}\right)+(1-a_{1})\left(\frac{1}{\alpha}+\frac{\mu}{Q}\right)
\end{equation}
and
\begin{equation}\label{2.12}
\frac{1}{\tau_{2}}+\frac{\gamma_{2}}{Q}= a_{2}\left(\frac{1}{p}+\frac{\beta-s}{Q}\right)+(1-a_{2})\left(\frac{1}{\alpha}+\frac{\mu}{Q}\right).
\end{equation}
Let $a_{1}$ and $a_{2}$ be such that
\begin{equation}\label{2.14}
|a-a_{1}|\,\,\,\text{and}\,\,\,|a-a_{2}|\,\,\, \text{are small enough},
\end{equation}
\begin{equation}\label{2.15}
a_{2}<a<a_{1},\,\,\,\text{if}\,\,\,\frac{1}{p}+\frac{\beta-s}{Q}>\frac{1}{\alpha}+\frac{\mu}{Q},\,\,\,
\end{equation}
\begin{equation}\label{2.16}
a_{1}<a<a_{2},\,\,\,\text{if}\,\,\,\frac{1}{p}+\frac{\beta-s}{Q}<\frac{1}{\alpha}+\frac{\mu}{Q}.
\end{equation}
By using \eqref{2.14}-\eqref{2.16} in \eqref{2.11}, \eqref{2.12} and \eqref{1.2}, we establish
\begin{equation}
\frac{1}{\tau_{1}}+\frac{\gamma_1}{Q}>\frac{1}{\tau}+\frac{\gamma}{Q}>\frac{1}{\tau_{2}}+\frac{\gamma_2}{Q}>0.
\end{equation}
From \eqref{tau1} in the case $\frac{a}{p}+\frac{1-a}{\alpha}-\frac{as}{Q}>0$ with $a>0$, $\beta-\sigma>s$ and  \eqref{2.14}, we get
\begin{equation}\label{2.19}
\frac{1}{\tau}-\frac{1}{\tau_{1}}=(a-a_{1})\left(\frac{1}{p}-\frac{s}{Q}-\frac{1}{\alpha}\right)+\frac{a}{Q}(\beta-\sigma)>0,
\end{equation}
and
\begin{equation}\label{2.18}
\frac{1}{\tau}-\frac{1}{\tau_{2}}=(a-a_{2})\left(\frac{1}{p}-\frac{1}{\alpha}\right)+\frac{a}{Q}(\beta-\sigma-s)>0.
\end{equation}
From \eqref{tau1}, \eqref{2.19} and \eqref{2.18}, we have
$$\tau_{1}>\tau,\,\,\,\tau_{2}>\tau.$$
Thus, using this, \eqref{2.14} and H\"{o}lder's inequality, we obtain
\begin{equation}
\|q^{\gamma}(x)u\|_{L^{\tau}(\mathbb{G}\setminus B_{1})}\leq C \|q^{\gamma_{1}}(x)u\|_{L^{\tau_{1}}(\mathbb{G})},
\end{equation}
and
\begin{equation}
\|q^{\gamma}(x)u\|_{L^{\tau}(B_{1})}\leq C \|q^{\gamma_{2}}(x)u\|_{L^{\tau_{2}}(\mathbb{G})},
\end{equation}
where $B_{1}$ is the unit quasi-ball.
By using the previous case, we establish
\begin{equation}
 \|q^{\gamma_{1}}(x)u\|_{L^{\tau_{1}}(\mathbb{G})}\leq C[u]^{a_{1}}_{s,p,\beta,q}\|q^{\mu}(x)u\|^{1-a_{1}}_{L^{\alpha}(\mathbb{G})}\leq C,
\end{equation}
and
\begin{equation}
 \|q^{\gamma_{2}}(x)u\|_{L^{\tau_{2}}(\mathbb{G})}\leq C[u]^{a_{2}}_{s,p,\beta,q}\|q^{\mu}(x)u\|^{1-a_{2}}_{L^{\alpha}(\mathbb{G})}\leq C.
\end{equation}
The proof of Theorem \ref{CKN1} is complete.
\end{proof}

\begin{rem}
By taking in \eqref{2in} $a=1$, $\tau = p$, $\beta_{1}=\beta_{2}=0$, and $\gamma = -s$, we get an analogue of the fractional Hardy inequality on homogeneous Lie groups (Theorem 2.9, \cite{KS}).
\end{rem}
\begin{rem}
In the Abelian case $(\mathbb{R}^{N},+)$ with the standard Eucledian distance instead of the quasi-norm and by taking in \eqref{2in} $a=1$, $\tau = p$, $\beta_{1}=\beta_{2}=0$, and $\gamma = -s$, we get the fractional Hardy inequality (Theorem 1.1, \cite{FS}).
\end{rem}
Now we consider the critical case $\frac{1}{\tau}+\frac{\gamma}{Q}=0.$
\begin{thm}\label{CKN2}
Assume that $Q\geq2$, $s\in(0,1)$, $p>1$, $\alpha\geq1$, $\tau>1$, $a\in(0,1]$, $\beta_1,\,\beta_2,\,\beta,\,\mu,\,\,\gamma\in\mathbb{R}$, $\beta_1+\beta_2=\beta$,

\begin{equation}\label{4.45}
\frac{1}{\tau}+\frac{\gamma}{Q}=a\left(\frac{1}{p}+\frac{\beta-s}{Q}\right)+(1-a)\left(\frac{1}{\alpha}+\frac{\mu}{Q}\right).
\end{equation}
Assume in addition that, $0\leq\beta-\sigma\leq s$ with
 $\gamma=a\sigma+(1-a)\mu.$

If $\frac{1}{\tau}+\frac{\gamma}{Q}=0$ and $\text{supp}\,u\subset B_{R},$ then, we have
\begin{equation}\label{3in}
\left\|\frac{q^{\gamma}(x)}{\ln\frac{2R}{q(x)}}u\right\|_{L^{\tau}(\mathbb{G})}\leq C [u]^{a}_{s,p,\beta,q}\|q^{\mu}(x)u\|^{1-a}_{L^{\alpha}(\mathbb{G})},\,\,u\in C_{c}^{1}(\mathbb{G}),
\end{equation}
where $B_R=\{x\in\mathbb{G}:q(x)<R\}$ is the quasi-ball and $0<r<R$.
\end{thm}

\begin{proof}[Proof of Theorem \ref{CKN2}]
The proof is similar to the proof of Theorem \ref{CKN1}. In \eqref{2.4}, summing over $k$ from $m$ to $n$ and fixing $\varepsilon>0$, we have \begin{multline}\label{3.1}
\int_{\{q(x)>2^{m}\}}\frac{q^{\gamma\tau}(x)}{\ln^{1+\varepsilon}\left(\frac{2R}{q(x)}\right)}|u|^{\tau}dx\leq C \sum_{k=m}^{n}\frac{1}{(n+1-k)^{1+\varepsilon}}|u_{A_{k,q}}|^{\tau}\\
+C \sum_{k=m}^{n} [u]^{a\tau}_{s,p,\beta,q,A_{k,q}}\|q^{\mu}(x)u\|^{(1-a)\tau}_{L^{\alpha}(A_{k,q})}.
\end{multline}
From Lemma \ref{GN1}, we have
\begin{equation*}
|u_{A_{k+1,q}}-u_{A_{k,q}}|\leq C2^{\frac{a k(sp-Q)}{p}}[u]_{s,p,q,A_{k+1,q}\cup A_{k,q}}^{a}\left(\mint{-}_{A_{k+1,q}\cup A_{k,q}}|u|^{\alpha}dx\right)^{\frac{1-a}{\alpha}}.
\end{equation*}
By using Lemma \ref{elin} with $\zeta=\frac{(n+1-k)^{\varepsilon}}{(n+\frac{1}{2}-k)^{\varepsilon}}$ we get
\begin{multline}\label{jojoj}
\frac{|u_{A_{k,q}}|^{\tau}}{(n+1-k)^{\varepsilon}}\leq \frac{|u_{A_{k+1,q}}|^{\tau}}{(n+\frac{1}{2}-k)^{\varepsilon}}\\
+C (n+1-k)^{\tau-1-\varepsilon}[u]^{a\tau}_{s,p,\beta,q,A_{k+1,q}\cup A_{k,q}}\|q^{\mu}(x)u\|^{(1-a)\tau}_{L^{\alpha}(A_{k+1,q}\cup A_{k,q})}.
\end{multline}
For $\varepsilon>0$ and $n\geq k,$ we have
\begin{equation}\label{3.2.2}
\frac{1}{(n-k+1)^{\varepsilon}}-\frac{1}{(n-k+\frac{3}{2})^{\varepsilon}}\sim \frac{1}{(n-k+1)^{1+\varepsilon}}.
\end{equation}
By using this fact, \eqref{jojoj}, \eqref{3.2.2} and $\varepsilon=\tau-1$, we obtain
\begin{equation}\label{3.4}
\sum^{n}_{k=m}\frac{|u_{A_{k,q}}|^{\tau}}{(n+1-k)^{\tau}}\leq C \sum^{n}_{k=m}[u]^{a\tau}_{s,p,\beta,q,A_{k+1,q}\cup A_{k,q}}\|q^{\mu}(x)u\|^{(1-a)\tau}_{L^{\alpha}(A_{k+1,q}\cup A_{k,q})}.
\end{equation}
From \eqref{3.1} and \eqref{3.4}, we establish
\begin{equation}
\int_{\{q(x)>2^{m}\}}\frac{q^{\gamma\tau}(x)}{\ln^{\tau}\frac{2R}{q(x)}}|u|^{\tau}dx\leq C \sum^{n}_{k=m}[u]^{a\tau}_{s,p,\beta,q,A_{k+1,q}\cup A_{k,q}}\|q^{\mu}(x)u\|^{(1-a)\tau}_{L^{\alpha}(A_{k+1,q}\cup A_{k,q})}.
\end{equation}
By using \eqref{kb} with \eqref{4.45} and $0\leq\beta-\sigma\leq s$, where $s=\frac{\tau a}{p}$, $t=\frac{(1-a)\tau}{\alpha}$, we have $s+t\geq1$ and we arrive at
\begin{equation}
\int_{\{q(x)>2^{m}\}}\frac{q^{\gamma\tau}(x)}{\ln^{\tau}\frac{2R}{q(x)}}|u|^{\tau}dx\leq C \sum^{n}_{k=m}[u]^{a\tau}_{s,p,\beta,q,\cup_{k=m}^{\infty} A_{k,q}}\|q^{\mu}(x)u\|^{(1-a)\tau}_{L^{\alpha}(\cup_{k=m}^{\infty} A_{k,q})}.
\end{equation}
Theorem \ref{CKN2} is proved.
\end{proof}

\section{Lyapunov-type inequalities for the fractional operators on $\mathbb{G}$}
\label{SEC:5}

 In this section we prove the Lyapunov-type  inequality for the Riesz potential and for the fractional $p$-sub-Laplacian system on homogeneous Lie groups. Note that the Lyapunov-type inequality for the Riesz operator is  new even in the Abelian case $(\mathbb{R}^{N},+)$. Also, we give applications of the Lyapunov-type inequality, more precisely,  we give two side estimates for the first eigenvalue of the Riesz potential of the fractional $p$-sub-Laplacian system.

Let us consider the Riesz potential on a Haar measurable set $\Omega\subset\mathbb{G}$ that can be defined by the formula
  \begin{equation}\label{rieszosn}
  \mathfrak{R}u(x)=\int_{\Omega}\frac{u(y)}{q^{Q-2s}(y^{-1}\circ x)}dy,\,\,\, 0<2s<Q.
  \end{equation}
The (weighted) Riesz potential can be also defined by
\begin{equation}\label{rieszosnw}
\mathfrak{R}(\omega u)(x)=\int_{\Omega}\frac{\omega(y) u(y)}{q^{Q-2s}(y^{-1}\circ x)}dy,\,\,\, 0<2s<Q.
\end{equation}
 \begin{thm}\label{Lyapunov}
 Let $\Omega\subset\mathbb{G}$ be a Haar measurable set and let $Q\geq2>2s>0$ and let $1< p< 2$.
Assume that  $\omega\in L^{\frac{p}{2-p}}(\Omega)$, $\frac{1}{q^{Q-2s}(y^{-1}\circ x)}\in L^{\frac{p}{p-1}}(\Omega\times\Omega)$ and $C_{0}=\left\|\frac{1}{q^{Q-2s}(y^{-1}\circ x)}\right\|_{L^{\frac{p}{p-1}}(\Omega\times\Omega)}$. Let $u\in L^{\frac{p}{p-1}}(\Omega)$, $u\neq0$, satisfy
\begin{equation}\label{r1}
\mathfrak{R}(\omega u)(x)=\int_{\Omega}\frac{\omega(y) u(y)}{q^{Q-2s}(y^{-1}\circ x)}dy=u(x),\,\,\text{for\,\,a.e.}\,\,x\in\Omega.
\end{equation}
Then
 \begin{equation}
\|\omega\|_{L^{\frac{p}{2-p}}(\Omega)}\geq \frac{1}{C_{0}}.
 \end{equation}
 \end{thm}
  \begin{proof}[Proof of Theorem \ref{Lyapunov}]
In \eqref{r1}, by using H\"{o}lder's inequality for $p,\theta> 1$ with $ \frac{1}{p}+\frac{1}{p'}=1\,\,\text{and}\,\,\frac{1}{\theta}+\frac{1}{\theta'}=1$, we have
\begin{multline}
|u(x)|=\left|\int_{\Omega}\frac{\omega(y) u(y)}{q^{Q-2s}(y^{-1}\circ x)}dy\right|\leq\left(\int_{\Omega}|\omega(y) u(y)|^{p}dy\right)^{\frac{1}{p}}\left(\int_{\Omega}\left|\frac{1}{q^{Q-2s}(y^{-1}\circ x)}\right|^{p'}dy\right)^{\frac{1}{p'}}\\
\leq\left(\int_{\Omega}|\omega(y)|^{p\theta}dy\right)^{\frac{1}{p\theta}}\left(\int_{\Omega}|u(y)|^{\theta'p}dy\right)^{\frac{1}{\theta'p}}\left(\int_{\Omega}\left|\frac{1}{q^{Q-2s}(y^{-1}\circ x)}\right|^{p'}dy\right)^{\frac{1}{p'}}\\
=\|\omega\|_{L^{p\theta}(\Omega)}\|u\|_{L^{p\theta'}(\Omega)}\left(\int_{\Omega}\left|\frac{1}{q^{Q-2s}(y^{-1}\circ x)}\right|^{p'}dy\right)^{\frac{1}{p'}}.
\end{multline}
Let $p'$  be such that $p'=p\theta'$ and then $\theta=\frac{1}{2-p}$. Thus, we get
\begin{equation}\label{rlyapposl}
|u(x)|\leq\|\omega\|_{L^{\frac{p}{2-p}}(\Omega)}\|u\|_{L^{\frac{p}{p-1}}(\Omega)}\left(\int_{\Omega}\left|\frac{1}{q^{Q-2s}(y^{-1}\circ x)}\right|^{\frac{p}{p-1}}dy\right)^{\frac{p-1}{p}}.
\end{equation}
From \eqref{rlyapposl} we calculate
\begin{equation*}
\|u\|_{L^{\frac{p}{p-1}}(\Omega)}\leq \|\omega\|_{L^{\frac{p}{2-p}}(\Omega)}\|u\|_{L^{\frac{p}{p-1}}(\Omega)}\left(\int_{\Omega}\int_{\Omega}\left|\frac{1}{q^{Q-2s}(y^{-1}\circ x)}\right|^{\frac{p}{p-1}}dxdy\right)^{\frac{p-1}{p}}
\end{equation*}
\begin{equation}
= C_{0} \|\omega\|_{L^{\frac{p}{2-p}}(\Omega)}\|u\|_{L^{\frac{p}{p-1}}(\Omega)}.
\end{equation}
Finally, since $u\neq0$, this implies
\begin{equation}
\|\omega\|_{L^{\frac{p}{2-p}}(\Omega)}\geq \frac{1}{C_{0}}.
\end{equation}
Theorem \ref{Lyapunov} is proved.
  \end{proof}

   Let us consider the following spectral problem for the Riesz potential:
   \begin{equation}\label{rieszspec}
   \mathfrak{R}u(x)=\int_{\Omega}\frac{u(y)}{q^{Q-2s}(y^{-1}\circ x)}dy=
   \lambda u(x),\,\,\,\,\,\,x\in\Omega,\,\,\, 0<2s<Q.
   \end{equation}
  We recall the Rayleigh quotient for the Riesz potential:

  \begin{equation}\label{Rayleigh}
  \lambda_{1}(\Omega)=\sup_{u\neq0}\frac{\int_{\Omega}\int_{\Omega}\frac{u(x)u(y)}{q^{Q-2s}(y^{-1}\circ x)}dxdy}{\|u\|^{2}_{L^{2}(\Omega)}},
  \end{equation}
    where $\lambda_{1}(\Omega)$ is the first eigenvalue of the Riesz potential.

So, a direct consequence of Theorem \ref{Lyapunov} is
   \begin{thm}\label{Lyapunapp}
 Let $\Omega\subset\mathbb{G}$ be a Haar measurable set and $Q\geq2>2s>0$ and let $1< p<2$.
Assume that  $\frac{1}{q^{Q-2s}(y^{-1}\circ x)}\in L^{\frac{p}{p-1}}(\Omega\times\Omega)$.
 Then for the spectral problem \eqref{rieszspec}, we have
 \begin{equation}
\lambda_{1}(\Omega)\leq C_{0}|\Omega|^{\frac{2-p}{p}},
 \end{equation}
 where $C_{0}=\left\|\frac{1}{q^{Q-2s}(y^{-1}\circ x)}\right\|_{L^{\frac{p}{p-1}}(\Omega\times\Omega)}$.
  \end{thm}
 \begin{proof}[Proof of Theorem \ref{Lyapunapp}]
By using \eqref{Rayleigh}, Theorem \ref{Lyapunov} and $\omega=\frac{1}{\lambda_{1}(\Omega)}$, we obtain
 \begin{equation}
 \lambda_{1}(\Omega)\leq C_{0} |\Omega|^\frac{2-p}{p}.
 \end{equation}
 Theorem \ref{Lyapunapp} is proved.
 \end{proof}

 In  the Abelian group $(\mathbb{R}^{N},+)$ we have the following consequences. To the best of our knowledge, these results seem new (even in this Euclidean case).

Let us consider the Riesz potential on $\Omega\subset\mathbb{R}^{N}$:
\begin{equation}\label{rieszspecRN}
  \mathfrak{R}u(x)=\int_{\Omega}\frac{u(y)}{|x-y|^{N-2s}}dy,\,\,\, 0<2s<N,
  \end{equation}
  and the weighted Riesz potential
 \begin{equation}\label{rieszspecRNwei}
  \mathfrak{R}(\omega u)(x)=\int_{\Omega}\frac{\omega(y)u(y)}{|x-y|^{N-2s}}dy,\,\,\, 0<2s<N.
  \end{equation}
  Then we have following theorem:
\begin{thm}\label{LyapunovRN}
 Let $\Omega\subset\mathbb{R}^{N},\,\,\,N\geq2,$ be a measurable set with $|\Omega|<\infty$, $1<p< 2$ and let  $N\geq2>2s>0$.
Assume that  $\omega\in L^{\frac{p}{2-p}}(\Omega)$, $\frac{1}{|x-y|^{N-2s}}\in L^{\frac{p}{p-1}}(\Omega\times\Omega)$ and let $S=\left\|\frac{1}{|x-y|^{N-2s}}\right\|_{L^{\frac{p}{p-1}}(\Omega\times\Omega)}$.
Assume that $u\in L^{\frac{p}{p-1}}(\Omega)$, $u\neq0$, satisfies
$$\mathfrak{R}(\omega u)(x)=u(x),\,\,\,x\in \Omega.$$
Then
 \begin{equation}
\|\omega\|_{L^{\frac{p}{2-p}}(\Omega)}\geq \frac{1}{S}.
 \end{equation}

 \end{thm}
  \begin{proof}[Proof of Theorem \ref{LyapunovRN}]
In Theorem \ref{Lyapunov} we set  $\mathbb{G}=(\mathbb{R}^{N},+)$ and take the standard Euclidean distance instead of the quasi-norm.
 \end{proof}
Let us consider the spectral problem for \eqref{rieszspecRN}:
\begin{equation}\label{rieszspecRNlambda}
  \mathfrak{R}u(x)=\int_{\Omega}\frac{u(y)}{|x-y|^{N-2s}}dy=\lambda u(x),\,\,\, 0<2s<N,
  \end{equation}

 \begin{thm}\label{LyapunappRN}
 Let $\Omega\subset\mathbb{R}^{N},\,N\geq2,$ be a set with $|\Omega|<\infty$, $1<p<2$ and $N\geq2>2s>0$ and $1<p< 2$. Assume that  $\omega\in L^{\frac{p}{2-p}}(\Omega)$, $\frac{1}{|x-y|^{N-2s}}\in L^{\frac{p}{p-1}}(\Omega\times\Omega)$ and $S=\left\|\frac{1}{|x-y|^{N-2s}}\right\|_{L^{\frac{p}{p-1}}(\Omega\times\Omega)}$. Then for the spectral problem \eqref{rieszspecRNlambda}
   we have,
 \begin{equation}
 \lambda_{1}(\Omega)\leq\lambda_{1}(B)\leq S|B|^{\frac{2-p}{p}},
 \end{equation}
 where $B\subset\mathbb{R}^{N}$ is an open ball, $\lambda_{1}(\Omega)$ is the first eigenvalue of the spectral problem \eqref{rieszspecRNlambda} with $|\Omega|=|B|.$
  \end{thm}
\begin{proof}[Proof of Theorem \ref{LyapunappRN}]
The proof of $\lambda_{1}(B)\leq S|B|^{\frac{2-p}{p}}$ is the same as the proof of Theorem \ref{Lyapunapp}. From \cite{RRS} we have
$$\lambda_{1}(B)\geq\lambda_{1}(\Omega).$$
The proof of Theorem \ref{LyapunappRN} is complete.
\end{proof}

In \cite{KS} the authors proved a Lyapunov-type inequality for the fractional $p$-sub-Laplacian with the homogeneous Dirichlet condition. Here we establish Lyapunov-type inequality for the fractional $p$-sub-Laplacian system for the homogeneous Dirichlet problem. Namely, let us consider the fractional $p$-sub-Laplacian system:
\begin{equation}\label{sys}
 \begin{cases}
   (-\Delta_{p_{1},q})^{s_{1}}u_{1}(x)=\omega_{1}(x)|u_{1}(x)|^{\alpha_{1}-2}u_{1}(x)|u_{2}(x)|^{\alpha_{2}}\ldots|u_{n}(x)|^{\alpha_{n}},\,\,x\in\Omega,\\
   (-\Delta_{p_{2},q})^{s_{2}}u_{2}(x)=\omega_{2}(x)|u_{1}(x)|^{\alpha_{1}}|u_{2}(x)|^{\alpha_{2}-2}u_{2}(x)\ldots|u_{n}(x)|^{\alpha_{n}},\,\,x\in\Omega,\\
   \quad\quad\quad\quad\quad\quad\quad\ldots\\
   (-\Delta_{p_{n},q})^{s_{n}}u_{n}(x)=\omega_{n}(x)|u_{1}(x)|^{\alpha_{1}}|u_{2}(x)|^{\alpha_{2}}\ldots|u_{n}(x)|^{\alpha_{n}-2}u_{n}(x),\,\,x\in\Omega,
 \end{cases}
\end{equation}
with homogeneous Dirichlet conditions
\begin{equation}\label{Dirprob}
 u_{i}(x)=0,\,\,\,\,x\in\mathbb{G}\setminus\Omega,\,\,\,i=1,\ldots,n,
\end{equation}
where $\Omega\subset\mathbb{G}$ is a Haar measurable set, $\omega_{i}\in L^{1}(\Omega)$, $\omega_{i}\geq0$, $s_{i}\in(0,1)$,  $p_{i}\in(1,\infty)$ and
 $(-\Delta_{p,q})^{s}$ is the fractional $p$-sub-Laplacian  on $\mathbb{G}$  defined by
\begin{multline}
(-\Delta_{p_{i},q})^{s_{i}}u_{i}(x)=2\lim_{\delta\searrow 0}\int_{\mathbb{G} \setminus B_{q}(x,\delta)}\frac{|u_{i}(x)-u_{i}(y)|^{p_{i}-2}(u_{i}(x)-u_{i}(y))}{q^{Q+s_{i}p_{i}}(y^{-1}\circ x)}dy, \,\,\,x\in \mathbb{G},\\ \,\,i=1,\ldots,n.
\end{multline}
 Here $B_{q}(x,\delta)$ is a quasi-ball with respect to $q$, with radius $\delta,$ centred at $x\in\mathbb{G},$ and $\alpha_{i}$ are positive parameters such that
 \begin{equation}\label{sum}
 \sum_{i=1}^{n}\frac{\alpha_{i}}{p_{i}}=1.
 \end{equation}
To prove a Lyapunov-type inequality for the system we need some preliminary results from \cite{KS}, the so-called fractional Hardy inequality on the homogeneous Lie groups.
\begin{thm}[\cite{KS}, Fractional Hardy inequality]\label{harun1}
For all $u\in C^{\infty}_{c}(\mathbb{G})$ we have
\begin{equation}\label{harun}
C\int_{\mathbb{G}}\frac{|u(x)|^{p}}{q^{ps}(x)}dx\leq[u]^{p}_{s,p,q},
\end{equation}
where $p\in(1,\infty),\,s\in(0,1),$ and $C$ is a positive constant.
\end{thm}
We denote by $r_{\Omega,q}$ the inner quasi-radius of $\Omega$, that is,
\begin{equation}
r_{\Omega,q}=\max\{q(x):\,\,x\in\Omega\}.
\end{equation}
\begin{defn}
We say that $(u_{1},\ldots,u_{n})\in \prod_{i=1}^{n} W^{s_{i},p_{i}}_{0} (\Omega)$ is a weak solution of \eqref{sys}-\eqref{Dirprob} if for all $(v_{1},\ldots,v_{n})\in \prod_{i=1}^{n} W^{s_{i},p_{i}}_{0} (\Omega)$, we have
\begin{multline}\label{5.6}
\int_{\mathbb{G}}\int_{\mathbb{G}}\frac{|u_{i}(x)-u_{i}(y)|^{p_{i}-2}(u_{i}(x)-u_{i}(y))(v_{i}(x)-v_{i}(y))}{q^{Q+s_{i}p_{i}}(y^{-1}\circ x)}dxdy\\
=\int_{\Omega}\omega_{i}(x)\left(\prod_{j=1}^{i-1}|u_{j}(x)|^{\alpha_{j}}\right)\left(\prod_{j=i+1}^{n}|u_{j}(x)|^{\alpha_{j}}\right)|u_{i}(x)|^{\alpha_{i}-2}u_{i}(x)v_{i}(x)dx,
\end{multline}
 for every $i=1,\ldots,n$.
\end{defn}
Now we present the following analogue of the Lyapunov-type inequality for the fractional $p$-sub-Laplacian system on $\mathbb{G}$.
\begin{thm}\label{thmlyap}
Let $s_{i}\in(0,1)$ and $p_{i}\in (1,\infty)$ be such that $Q>s_{i}p_{i}$ for all $i=1,\ldots,n$. Let $\omega_{i}\in L^{\theta}(\Omega)$  be a non-negative weight and assume that
$$1<\max_{i=1,...,n}\left\{\frac{Q}{s_{i}p_{i}}\right\}<\theta<\infty.$$
If \eqref{sys}-\eqref{Dirprob} admits a nontrivial weak solution, then
\begin{equation}\label{lyap}
\prod_{i=1}^{n}\|\omega_{i}\|^{\frac{\theta\alpha_{i}}{p_{i}}}_{L^{\theta}(\Omega)}\geq C r_{\Omega,q}^{Q-\theta\sum_{j=1}^{n}s_{j}\alpha_{j}},
\end{equation}
where $C>0$ is a positive constant.
\end{thm}
\begin{rem}
	In Theorem \ref{thmlyap}, by taking $n=1$ and $\alpha_{1}=p$, we establish the Lyapunov-type inequality for the fractional $p$-sub-Laplacian on $\mathbb{G}$ (see, e.g. \cite[Theorem 3.1]{KS}).
\end{rem}
\begin{proof}[Proof of Theorem \ref{thmlyap}]
For all $i=1,\ldots,n,$ let us define
\begin{equation}\label{xij}
\xi_{i}=\gamma_{i} p_{i}+(1-\gamma_{i})p_{i}^{*},
\end{equation}
and
\begin{equation}\label{gammaj}
\gamma_{i}=\frac{\theta-\frac{Q}{s_{i}p_{i}}}{\theta-1},
\end{equation}
where $p_{i}^{*}=\frac{Q}{Q-s_{i}p_{i}}$ is the Sobolev conjugate exponent as in Theorem \ref{sob}. Notice that for all  $i=1,\ldots,n$ we have $\gamma_{i}\in(0,1)$ and $\xi_{i}=p_{i}\theta'$, where $\theta'=\frac{\theta}{\theta-1}$. Then for every $i\in\{1,\ldots,n\}$ we get
$$\int_{\Omega}\frac{|u_{i}(x)|^{\xi_{i}}}{r^{\gamma_{i}s_{i}p_{i}}_{\Omega,q}}dx\leq\int_{\Omega}\frac{|u_{i}(x)|^{\xi_{i}}}{q^{\gamma_{i}s_{i}p_{i}}(x)}dx,$$
and by using H\"{o}lder's inequality with the following exponents $\nu_{i}=\frac{1}{\gamma_{i}}$ and $\frac{1}{\nu_{i}}+\frac{1}{\nu'_{i}}=1,$ we get
\begin{multline}\label{5.24}
\int_{\Omega}\frac{|u_{i}(x)|^{\xi_{i}}}{q^{\gamma_{i}s_{i}p_{i}}(x)}dx=\int_{\Omega}\frac{|u_{i}(x)|^{\gamma_{i}p_{i}}|u_{i}(x)|^{(1-\gamma_{i})p_{i}^{*}}}{q^{\gamma_{i}s_{i}p_{i}}(x)}dx\\
\leq\left(\int_{\Omega}\frac{|u_{i}(x)|^{p_{i}}}{q^{s_{i}p_{i}}(x)}dx\right)^{\gamma_{i}}\left(\int_{\Omega}|u_{i}(x)|^{p^{*}_{i}}dx\right)^{1-\gamma_{i}}.
\end{multline}
On the other hand, from Theorem \ref{sob}, we obtain
$$\left(\int_{\Omega}|u_{i}(x)|^{p^{*}_{i}}dx\right)^{1-\gamma_{i}}\leq C [u_{i}]_{s_{i},p_{i},q}^{p^{*}_{i}(1-\gamma_{i})},$$
and from Theorem \ref{harun1}, we have
$$\left(\int_{\Omega}\frac{|u_{i}(x)|^{p_{i}}}{q^{s_{i}p_{i}}(x)}dx\right)^{\gamma_{i}}\leq C[u_{i}]_{s_{i},p_{i},q}^{p_{i}\gamma_{i}}.$$
Thus, from \eqref{5.24} and by taking $u_{i}(x)=v_{i}(x)$ in \eqref{5.6}, we get
$$\int_{\Omega}\frac{|u_{i}(x)|^{\xi_{i}}}{q^{\gamma_{i}s_{i}p_{i}}(x)}\leq C ([u_{i}]_{s_{i},p_{i},q,\Omega}^{p_{i}})^{\frac{\xi_{i}}{p_{i}}}\leq C ([u_{i}]_{s_{i},p_{i},q}^{p_{i}})^{\frac{\xi_{i}}{p_{i}}}$$
$$=C\left(\int_{\Omega}\omega_{i}(x)\prod_{j=1}^{n}|u_{j}|^{\alpha_{j}}dx\right)^{\frac{\xi_{i}}{p_{i}}}=C \left(\int_{\Omega}\omega_{i}(x)\prod_{j=1}^{n}|u_{j}|^{\alpha_{j}}dx\right)^{\theta'},$$
for every $i=1,\ldots,n$. Therefore, by using H\"{o}lder's inequality with exponents $\theta$ and $\theta'$, we obtain
$$\int_{\Omega}\frac{|u_{i}(x)|^{\xi_{i}}}{q^{\gamma_{i}s_{i}p_{i}}(x)}dx\leq C\|\omega_{i}\|^{\frac{\theta}{\theta-1}}_{L^{\theta}(\Omega)}\int_{\Omega}\prod_{j=1}^{n}|u_{j}(x)|^{\alpha_{j}\theta'}dx.$$
By using H\"{o}lder's inequality and \eqref{sum}, we get
$$\int_{\Omega}\prod_{j=1}^{n}|u_{j}(x)|^{\alpha_{j}\theta'}dx\leq \prod_{j=1}^{n}\left(\int_{\Omega}|u_{j}|^{\theta'p_{j}}dx\right)^{\frac{\alpha_{j}}{p_{j}}}.$$
This implies that
$$\int_{\Omega}\frac{|u_{i}(x)|^{\xi_{i}}}{q^{\gamma_{i}s_{i}p_{i}}(x)}dx\leq C\|\omega_{i}\|^{\frac{\theta}{\theta-1}}_{L^{\theta}(\Omega)}\prod_{j=1}^{n}\left(\int_{\Omega}|u_{j}|^{\theta'p_{j}}dx\right)^{\frac{\alpha_{j}}{p_{j}}}.$$
So we establish
$$\int_{\Omega}\frac{|u_{i}(x)|^{\xi_{i}}}{r^{\gamma_{i}s_{i}p_{i}}_{\Omega,q}}dx\leq\int_{\Omega}\frac{|u_{i}(x)|^{\xi_{i}}}{q^{\gamma_{i}s_{i}p_{i}}(x)}dx$$
$$\leq C\|\omega_{i}\|^{\frac{\theta}{\theta-1}}_{L^{\theta}(\Omega)}\prod_{j=1}^{n}\left(\int_{\Omega}|u_{j}|^{\theta'p_{j}}dx\right)^{\frac{\alpha_{j}}{p_{j}}}.$$
Thus, for every $e_{i}>0$ we have
$$\left(\int_{\Omega}\frac{|u_{i}(x)|^{\xi_{i}}}{r^{\gamma_{i}s_{i}p_{i}}_{\Omega,q}}dx\right)^{e_{i}}=\frac{1}{r^{e_{i}\gamma_{i}s_{i}p_{i}}_{\Omega,q}}\left(\int_{\Omega}|u_{i}(x)|^{\xi_{i}}dx\right)^{e_{i}}$$
$$\leq C\|\omega_{i}\|^{\frac{e_{i}\theta}{\theta-1}}_{L^{\theta}(\Omega)}\prod_{j=1}^{n}\left(\int_{\Omega}|u_{j}|^{\theta'p_{j}}dx\right)^{\frac{e_{i}\alpha_{j}}{p_{j}}},$$
so that
$$\frac{1}{r_{\Omega,q}^{\sum_{j=1}^{n}\gamma_{j}s_{j}p_{j}e_{j}}}\prod_{i=1}^{n}\left(\int_{\Omega}|u_{i}(x)|^{\theta'p_{i}}dx\right)^{e_{i}}$$
$$\leq C \left(\prod_{i=1}^{n}\|\omega_{i}\|^{\frac{e_{i}\theta}{\theta-1}}_{L^{\theta}(\Omega)}\right)\left(\prod_{i=1}^{n}\left(\int_{\Omega}|u_{i}(x)|^{\theta'p_{i}}dx\right)^{\frac{\alpha_{i}\sum_{j=1}^{n}e_{j}}{p_{i}}}\right).$$
This yields
\begin{equation}\label{predpos}
\frac{1}{r_{\Omega,q}^{\sum_{j=1}^{n}\gamma_{j}s_{j}p_{j}e_{j}}}\leq C \left(\prod_{i=1}^{n} \|\omega_{i}\|^{\frac{e_{i}\theta}{\theta-1}}_{L^{\theta}(\Omega)}\right) \left(\prod_{i=1}^{n}\left(|u_{i}(x)|^{\theta'p_{i}}dx\right)^{\frac{\alpha_{i}\sum_{j=1}^{n}e_{j}}{p_{i}}-e_{i}}\right),
\end{equation}
where $C$ is a positive constant.
Then, we choose $e_{i},\,i=1,\ldots,n,$ such that
$$\frac{\alpha_{i}\sum_{j=1}^{n}e_{j}}{p_{i}}-e_{i}=0,\,\,\,i=1,\ldots,n.$$
Consequently, from \eqref{sum} we have the solution of this system
\begin{equation}\label{ej}
e_{i}=\frac{\alpha_{i}}{p_{i}},\,\,i=1,\ldots,n.
\end{equation}
From \eqref{predpos}, \eqref{gammaj} and \eqref{ej} we arrive at
\begin{equation}
\prod_{i=1}^{n}\|\omega_{i}\|^{\frac{\theta\alpha_{i}}{p_{i}}}_{L^{\theta}(\Omega)}\geq C r_{\Omega,q}^{Q-\theta\sum_{j=1}^{n}s_{j}\alpha_{j}}.
\end{equation}
Theorem \ref{thmlyap} is proved.
\end{proof}

Now, let us discuss an application of the Lyapunov-type inequality for the fractional $p$-sub-Laplacian system on $\mathbb{G}.$ In order to do it we consider the spectral problem for the fractional $p$-sub-Laplacian system in the following form:
\begin{equation}\label{sys1}
 \begin{cases}
   (-\Delta_{p_{1},q})^{s_{1}}u_{1}(x)=\lambda_{1}\alpha_{1}\varphi(x)|u_{1}(x)|^{\alpha_{1}-2}u_{1}(x)|u_{2}(x)|^{\alpha_{2}}\ldots|u_{n}(x)|^{\alpha_{n}},\,\,x\in\Omega,\\
   (-\Delta_{p_{2},q})^{s_{2}}u_{2}(x)=\lambda_{2}\alpha_{2}\varphi(x)|u_{1}(x)|^{\alpha_{1}}|u_{2}(x)|^{\alpha_{2}-2}u_{2}(x)\ldots|u_{n}(x)|^{\alpha_{n}},\,\,x\in\Omega,\\
   \quad\quad \quad \quad \quad \quad \quad  {\ldots}\\
   (-\Delta_{p_{n},q})^{s_{n}}u_{n}(x)=\lambda_{n}\alpha_{n}\varphi(x)|u_{1}(x)|^{\alpha_{1}}|u_{2}(x)|^{\alpha_{2}}\ldots|u_{n}(x)|^{\alpha_{n}-2}u_{n}(x),\,\,x\in\Omega,
 \end{cases}
\end{equation}
with
\begin{equation}\label{Dirprob1}
 u_{i}(x)=0,\,\,\,\,x\in\mathbb{G}\setminus\Omega,\,\,\,i=1,\ldots,n,
\end{equation}
where $\Omega\subset\mathbb{G}$ is a Haar measurable set, $\varphi\in L^{1}(\Omega)$, $\varphi\geq0$ and $s_{i}\in(0,1)$,  $p_{i}\in(1,\infty),\,\,i=1,\ldots,n.$
\begin{defn}
We say that $\lambda=(\lambda_{1},\ldots,\lambda_{n})$ is an eigenvalue if the problem \eqref{sys1}-\eqref{Dirprob1} admits at least one nontrivial weak solution $(u_{1},\ldots,u_{n})\in \prod_{i={1}}^{n} W^{s_{i},p_{i}}_{0}(\Omega).$
\end{defn}
\begin{thm}\label{thmapplyap}
Let $s_{i}\in(0,1)$ and $p_{i}\in(1,\infty)$ be such that $Q>s_{i}p_{i},$ for all $i=1,\ldots,n,$  and $$1<\max_{i=1,\ldots,n}\left\{\frac{Q}{s_{i}p_{i}}\right\}<\theta<\infty.$$
Let $\varphi\in L^{\theta}(\Omega)$ with $\|\varphi\|_{L^{\theta}(\Omega)}\neq0$.
Then, we have
\begin{equation}
\lambda_{k}\geq\frac{C}{\alpha_{k}}\left(\frac{1}{\prod_{i=1,i\neq k}^{n}\lambda_{i}^{\frac{\alpha_{i}}{p_{i}}}}\right)^\frac{p_{k}}{\alpha_{k}}\left(\frac{1}{r_{\Omega,q}^{\theta\sum_{i=1}^{n}\alpha_{i}s_{i}-Q}\prod_{i={1},i\neq k}^{n}\alpha_{i}^{\frac{\theta\alpha_{i}}{p_{i}}}\int_{\Omega}\varphi^{\theta}(x)dx}\right)^{\frac{p_{k}}{\theta\alpha_{k}}},
\end{equation}
where $C$ is a positive constant and $k=1,\ldots,n$.
\end{thm}
\begin{proof}[Proof of Theorem \ref{thmapplyap}]
In Theorem \ref{thmlyap} by taking $\omega_{k}=\lambda_{k}\alpha_{k}\varphi(x),\,\,k=1,\ldots,n$, we have
$$ \alpha_{k}^{\frac{\theta\alpha_{k}}{p_{k}}}\lambda_{k}^{\frac{\theta\alpha_{k}}{p_{k}}}\prod_{i=1,i\neq k}^{n}(\alpha_{i}\lambda_{i})^{\frac{\theta\alpha_{i}}{p_{i}}}\prod_{i=1}^{n}\|\varphi\|_{L^{\theta}(\Omega)}^{\frac{\theta \alpha_{i}}{p_{i}}}\geq C r_{\Omega,q}^{Q-\theta\sum_{j=1}^{n}s_{j}\alpha_{j}}.$$
Thus, using \eqref{sum} we obtain
$$\alpha_{k}^{\frac{\theta\alpha_{k}}{p_{k}}}\lambda_{k}^{\frac{\theta\alpha_{k}}{p_{k}}}\prod_{i=1,i\neq k}^{n}(\alpha_{i}\lambda_{i})^{\frac{\theta\alpha_{i}}{p_{i}}}\int_{\Omega}\varphi^{\theta}(x)dx    \geq C r_{\Omega,q}^{Q-\theta\sum_{j=1}^{n}s_{j}\alpha_{j}}.$$
This implies
$$\lambda_{k}^{\frac{\theta\alpha_{k}}{p_{k}}}\geq \frac{C}{\alpha_{k}^{\frac{\theta\alpha_{k}}{p_{k}}}r_{\Omega,q}^{\theta\sum_{j=1}^{n}s_{j}\alpha_{j}-Q}\prod_{i=1,i\neq k}^{n}(\alpha_{i}\lambda_{i})^{\frac{\theta\alpha_{i}}{p_{i}}}\int_{\Omega}\varphi^{\theta}(x)dx},\,\,k=1,\ldots,n.$$
Finally, we get that
\begin{multline}
\lambda_{k}\geq\frac{C}{\alpha_{k}}\left(\frac{1}{\prod_{i=1,i\neq k}^{n}\lambda_{i}^{\frac{\alpha_{i}}{p_{i}}}}\right)^\frac{p_{k}}{\alpha_{k}}\left(\frac{1}{r_{\Omega,q}^{\theta\sum_{i=1}^{n}\alpha_{i}s_{i}-Q}\prod_{i={1},i\neq k}^{n}\alpha_{i}^{\frac{\theta\alpha_{i}}{p_{i}}}\int_{\Omega}\varphi^{\theta}(x)dx}\right)^{\frac{p_{k}}{\theta\alpha_{k}}},\\
k=1,\ldots,n.
\end{multline}
Theorem \ref{thmapplyap} is proved.
\end{proof}


\begin{thebibliography}{H8}

\bibitem{Abd}
B.~Abdellaoui and R.~Bentifour.
\newblock Caffarelli-Kohn-Nirenberg type inequalities of fractional order with applications.
\newblock  {\em J. Funct. Anal.}, 272(10):3998--4029, 2017.

\bibitem{CKN}
L. A.~Caffarelli, R.~Kohn and L.~Nirenberg.
\newblock First order interpolation inequalities with weights.
\newblock  {\em Composito Math.}, 53(3):259--275, 1984.

\bibitem{CR}
J.~Chen and E. M.~Rocha.
\newblock  A class of sub-elliptic equations on the Heisenberg group and
related interpolation inequalities.
\newblock   {\em Oper. Theory Adv. Appl.},  229:123--137, Birkh\"{a}user/Springer Basel AG,
Basel, 2013.

\bibitem{DP}
P. L.~De N\'{a}poli, J. P.~Pinasco.
\newblock Estimates for eigenvalues of quasilinear elliptic systems.
\newblock {\em J. Differential Equations}, 227(10):102--115, 2006.

\bibitem{Ed}
A.~Elbert.
\newblock A half-linear second order differential equation.
\newblock   {\em Colloq. Math. Soc.}, Janos Bolyai, 30:158--180, 1979.
\bibitem{FR}
V.~Fischer and M.~Ruzhansky.
\newblock Quantization on nilpotent Lie groups.
\newblock Progress in Mathematics, Vol. 314, Birkhauser, 2016. (open access book)
\bibitem{FS1}
G. B.~Folland and E. M.~Stein.
\newblock Hardy spaces on homogeneous groups.
\newblock Mathematical Notes, Vol. 28, Princeton University Press, Princeton, N.J.; University of
Tokyo Press, Tokyo, 1982.
 \bibitem{FS}
 R.L.~Frank and R.~Seiringer.
\newblock Non-linear ground state representations
and sharp Hardy inequalities.
\newblock{\em J. Funct. Anal.},  255, 3407--3430, 2008.
\bibitem{Gag}
E.~Gagliardo.
\newblock Ulteriori propriet\`{a} di alcune classi di funzioni in pi\`{u} variabili.
\newblock {\em Ricerche Mat.}, 8:24--51, 1959.


\bibitem{Kir}
M.~Jleli, M.~Kirane and B.~ Samet.
\newblock Lyapunov-type inequalities for fractional partial differential equations.
\newblock {\em Appl. Math. Lett.}, 66:30--39, 2017.
\bibitem{Kir1}
M.~Jleli, M.~Kirane and B.~ Samet.
 \newblock Lyapunov-type inequalities for a fractional p-Laplacian system.
 \newblock {\em Fract. Calc. Appl. Anal.}, 20(6): 1485--1506, 2017.
\bibitem{KS}
A.~Kassymov and D.~Suragan.
\newblock Some functional inequalities for the fractional p-sub-Laplacian.
\newblock {\em arXiv:1804.01415v2 ,} 2018.

\bibitem{Lyap}
A. M.~Lyapunov.
\newblock Probl\`{e}me g\`{e}n\`{e}ral de la stabilit\`{e} du mouvement.
\newblock {\em Ann. Fac. Sci. Univ. Toulouse}, 2:203--407, 1907.

\bibitem{Nir}
L.~Nirenberg.
\newblock On elliptic partial differential equations.
\newblock {\em Ann. Scuola Norm. Sup. Pisa (3)}, 13:115--162, 1959.

\bibitem{NS1}
H.-M.~Nguyen, M.~Squassina.
\newblock Fractional Caffarelli-Kohn-Nirenberg inequalities.
\newblock  {\em J. Funct.
Anal.}, 274:2661--2672, 2018.

\bibitem{NS2}
H.-M.~Nguyen, M.~Squassina.
\newblock On Hardy and Caffarelli-Kohn-Nirenberg inequalities.
\newblock{\em arXiv:1801.06329}, 2018.
\bibitem{RRS}
G.~Rozenblum,  M.~Ruzhansky and D.~Suragan.
\newblock Isoperimetric inequalities for Schatten norms of Riesz potentials.
\newblock {\em J. Funct. Anal.}, 271:224--239, 2016.
\bibitem{RSAM}
M.~Ruzhansky and D.~Suragan.
 \newblock Hardy and Rellich inequalities, identities, and sharp remainders on homogeneous groups.
 \newblock{\em Adv. Math.}, 317, 799--822, 2017.
\bibitem{RSY1}
M.~Ruzhansky, D.~Suragan and N.~Yessirkegenov.
\newblock Extended Caffarelli-Kohn-Nirenberg inequalities, and remainders, stability, and superweights for $L^p$-weighted
Hardy inequalities.
\newblock {\em Trans. Amer. Math. Soc. Ser. B}, 5:32--62, 2018.
\bibitem{RTY}
M.~Ruzhansky, N.~Tokmagambetov, N.~Yessirkegenov,
\newblock Best constants in Sobolev and Gagliardo-Nirenberg inequalities on graded groups and ground states for higher order nonlinear subelliptic equations.
\newblock {\em arXiv:1704.01490v1}, 2018.
\bibitem{RSY2}
M.~Ruzhansky, D.~Suragan and N.~Yessirkegenov.
\newblock Extended Caffarelli-Kohn-Nirenberg inequalities and superweights for $L^p$-weighted Hardy inequalities.
\newblock {\em C. R. Math. Acad. Sci. Paris}, 355(6):694--698, 2017.
\bibitem{RSY3}
M.~Ruzhansky, D.~Suragan and N.~Yessirkegenov.
\newblock Caffarelli-Kohn-Nirenberg and Sobolev type inequalities on stratified Lie groups.
\newblock {\em NoDEA Nonlinear Differential Equations Appl.}, 24(5):Art. 56, 2017.
\bibitem{RS}
M.~Ruzhansky and D.~Suragan.
\newblock On horizontal Hardy, Rellich, Caffarelli-Kohn-Nirenberg and p-sub-Laplacian inequalities on stratified groups.
\newblock {\em J. Differential Equations}, 262:1799--1821, 2017.
\bibitem{ZHD}
S.~Zhang, Y.~Han and J.~Dou.
\newblock A class of Caffarelli-Kohn-Nirenberg type inequalities on the H-type group.
\newblock {\em Sem. Mat. Univ. Padova}, 132:249--266, 2014.

\end{thebibliography}
\end{document}